\documentclass[10pt,a4paper]{article}

\usepackage{amsmath,amssymb,amsthm,mathrsfs}
\usepackage{stmaryrd}  %\interleave
\usepackage{graphicx}
\usepackage{xcolor}
\usepackage{fullpage}
\usepackage{hyperref}
\newlength{\defbaselineskip}
\newcommand{\setlinespacing}[1]%
           {\setlength{\baselineskip}{#1 \defbaselineskip}}

\theoremstyle{plain}
\newtheorem{theorem}{Theorem}[section]
\newtheorem{lemma}[theorem]{Lemma}
\newtheorem{proposition}[theorem]{Proposition}
\newtheorem{corollary}[theorem]{Corollary}
\newtheorem{Fact}[theorem]{Fact}

\theoremstyle{definition}
\newtheorem{definition}[theorem]{Definition}

\newtheorem{ass}[theorem]{Assumption}

\theoremstyle{remark}
\newtheorem{remark}[theorem]{Remark}

\numberwithin{equation}{section}

\newcommand{\eps}{\varepsilon}

\DeclareMathOperator*{\esssup}{ess\,sup}

\newcommand{\cP}{\mathcal{P}}

\newcommand{\cH}{\mathcal{H}}

\newcommand{\cM}{\mathcal{M}}

\newcommand{\cF}{\mathcal{F}}

\newcommand{\cA}{\mathcal{A}}
\newcommand{\cX}{\mathcal{X}}

\newcommand{\fC}{\mathfrak{C}}

\newcommand{\fp}{\mathfrak{p}}
\newcommand{\fd}{\mathfrak{d}}

\newcommand{\bE}{\mathbb{E}}

\newcommand{\bP}{\mathbb{P}}
\newcommand{\bR}{\mathbb{R}}

\newcommand{\bF}{\mathbb{F}}
\newcommand{\bG}{\mathbb{G}}

\newcommand{\prog}{\textrm{Prog}}

\newcommand{\etamax}{\|\eta\|}%{\etamax}

\newcommand{\etamin}{\eta_\star}%{\etamin}

\newcommand{\lambdamin}{\lambda_\star}%{\lambdamin}

\newcommand{\kappamax}{\kappa_{\max}}%{\etamax}

\newcommand{\lambdamax}{\lambda_{\max}}%{\lambdamax}

 \allowdisplaybreaks
\begin{document}

\title{A Mean Field Game of Optimal Portfolio Liquidation.
%\thanks{Financial support by the Berlin Mathematical School (BMS) and the TRCRC 190 {\sl Rationality and competition: the economic performance of individuals and firms} is gratefully acknowledged. We thank participants of  the IPAM workshop ``Mean Field Games'' for valuable comments and suggestions. We also thank seminar participants at various places for helpful comments and discussions.}
}

\author{Guanxing Fu\footnote{Department of Applied Mathematics, The Hong Kong Polytechnic University, Hung Hom, Kowloon, Hong Kong; email:fuguanxing725@gmail.com}
\and Paulwin Graewe\footnote{Deloitte Consulting GmbH, Kurf\"urstendamm 23, 10719 Berlin, Germany; email:pgraewe@deloitte.de} \and Ulrich Horst\footnote{Department of Mathematics, and School of Business and Economics, Humboldt-Universit\"at zu Berlin,
         Unter den Linden 6, 10099 Berlin, Germany; email: horst@math.hu-berlin.de}   \and Alexandre Popier\footnote{Laboratoire Manceau de Math\'ematiques, Le Mans Universit\'e, Avenue Olivier Messiaen, 72058 Le Mans Cedex 9, France; email: Alexandre.Popier@univ-lemans.fr}}
%\address{Department of Mathematics, Humboldt-Universit\"at zu Berlin
%         Unter den Linden 6, 10099 Berlin, Germany.}
%
%\email{ fuguanxing725@gmail.com (Guanxing Fu), horst@math.hu-berlin.de (Ulrich Horst)
%}

\maketitle

\begin{abstract}
We consider a mean field game (MFG) of optimal portfolio liquidation under asymmetric information. We prove that the solution to the MFG can be characterized in terms of a FBSDE with possibly singular terminal condition on the backward component or, equivalently, in terms of a FBSDE with finite terminal value, yet singular driver. Extending the method of continuation to linear-quadratic FBSDE with singular driver we prove that the MFG has a unique solution. Our existence and uniqueness result allows to prove that the MFG with possibly singular terminal condition can be approximated by a sequence of MFGs with finite terminal values.
\end{abstract}

%\subjclass[2000]{60H15, 35R16}


\section{Introduction and overview}

Mean field games (MFGs) are a powerful tool to analyse strategic interactions in large populations when each individual player has only a small impact on the behavior of other players. In the economics literature, mean-field-type (or anonymous) games were first considered by Jovanovic and Rosenthal \cite{JR-1988} and later analyzed by many authors including \cite{B-1999, DP-2015, H-2005}. In the mathematical literature MFGs were independently introduced by Huang, Malham\'e and Caines \cite{HMC-2006} and Lasry and Lions \cite{LL-2007}. MFGs have been successfully applied to various economic problems, ranging from systemic risk management \cite{CFS-2015} to principal agent problems \cite{EMP-2016,NZ-2018} and from portfolio optimization \cite{LZ-2018} to optimal exploitation of exhaustible resources \cite{CS-2017}.

In a standard MFG, each player $i \in \{1, ..., N\}$ chooses an action $u^i$ from a given set of admissible controls that minimizes a  cost functional of the form
    \begin{equation}\label{cost-fun}
        J^i(\vec u)=\bE\left[\int_0^Tf(t,X^i_t,\bar{\mu}^N_t,u^i_t)dt+g(X^i_T,\bar{\mu}^N_T)  \big | \cX^i = x^i\right]
    \end{equation}
subject to the state dynamics
   \begin{equation}\label{state-without-singular}
    \left\{\begin{array}{ll}
        dX^i_t=b(t,X^i_t,\bar{\mu}^N_t,u^i_t)\,dt+\sigma(t,X^i_t,\bar{\mu}^N_t,u^i_t)\,dW^i_t,\\
        X^i_0={\cX^i}
        \end{array}.
        \right.
    \end{equation}
Here, $W^1,\cdots, W^N$ are independent Brownian motions, and $\cX^1, ..., \cX^N$ are independent and identically distributed random variables with law $\nu$ that are independent of the Brownian motions. All stochastic processes and random variables are defined on an underlying filtered probability space\footnote{We assume throughout that all filtrations are augmented by the null sets. }.
The vector $\vec u=(u^1,\cdots,u^N)$ denotes the action profile, and $\bar{\mu}^N_t:=\frac{1}{N}\sum_{j=1}^N\delta_{X^j_t}$ denotes the empirical distribution of the individual players' states at time $t \in [0,T]$. It is usually assumed that the players observe their own initial state and know the common distribution $\nu$ of the other player's initial states.

The existence of {\it approximate Nash equilibria} for large populations can be established using a representative agent approach.The idea is to approximate the dynamics of the empirical distribution of the states by a deterministic measure-valued process, and then to consider the optimization problem of a representative player subject to the equilibrium constraint that the distribution of the representative player's state process under her optimal strategy coincides with the pre-specified measure-valued process. More precisely, denoting by $\mathcal{P}(\mathbb{R}^d)$ the space of probability measures on $\mathbb{R}^d$, by $\textrm{Law}(X)$ the law of a stochastic process $X$ and by  $\mathcal{X}$ a random initial state with distribution $\nu$ the resulting MFG can be formally described as follows:
\begin{equation}\label{model-MFG}
        \left\{
        \begin{array}{ll}
        1.&\textrm{fix a deterministic function }t\in[0,T]\mapsto\mu_t\in\mathcal{P}(\mathbb{R}^d);\\
        2.&\textrm{solve the corresponding stochastic control problem}:\\
         &\quad \inf_{u} \bE\left[\left.\int_0^Tf(t,X^{\mathcal X}_t,\mu_t,u_t)\,dt+g(X^{\mathcal X}_T,\mu_T)\right|\mathcal X\right],\\
        &\textrm{subject to the state dynamics}\\
        &~~ dX^{\mathcal X}_t =b(t,X^{\mathcal X}_t,{\mu}_t,u_t)\,dt+\sigma(t,X^{\mathcal X}_t,{\mu}_t,u_t)\,dW_t, \\ & ~~~~~ X_0=\mathcal X \\
        3.&\textrm{solve } \textrm{Law}(X^{*,\mathcal X}) =\mu{\textrm{ where }X^{*,\mathcal X}\textrm{ is the optimal state process from 2}}.
        \end{array}.
        \right.
\end{equation}

Let $\mu^*$ be a solution to the above fix point problem and let $u^*$ be the representative player's optimal response to  $\mu^*$ given $\cX$. Then $u^* = \phi(\cX,W)$ for some measurable function $\phi$ from $\bR \times \mathcal C[0,T]$ into a suitable function space, and each individual player's optimal response to $\mu^*$ given her initial state $\cX^i = x^i$ is  $u^{*,i}=\phi(x^i,W^i)$. Under suitable assumptions the homogeneous action profile $\left( \phi(\cdot,\cdot), ..., \phi(\cdot,\cdot) \right)$ forms an $\epsilon$-equilibrium in the original game if $N$ is large enough.

There are basically four approaches to solve mean field games. In their original paper \cite{LL-2007}, Lasry and Lions followed an analytic approach. They analyzed a coupled forward-backward PDE system, where the backward component is the Hamiltion-Jacobi-Bellman equation arising from the representative agent's optimization problem, and the forward component is a Kolmogorov-Fokker-Planck equation that characterizes the dynamics of the state process; see also \cite{Gomes2013}.  Merging the forward backward system into a single master equation, the dynamics of the MFG can alternatively be described in terms of some form of second order PDE on the space of probability measures; see \cite{CDLL-2015,Car:Del-2014,Chassagneux2014,DLR-2018} for details. A more probabilistic approach was introduced by Carmona and Delarue in \cite{CD-2013}. Using a maximum principle of Pontryagin type, they showed that the fixed point problem reduces to solving a McKean-Vlasov FBSDEs; see also \cite{BSYY-2016,CZ-2016}.  A \textit{relaxed solution} concept to MFGs was introduced by Lacker in \cite{L-2015} and later extended by various authors including \cite{CDL-2016,FH-2017}. In this paper we apply a probabilistic approach to analyze a novel class of MFGs arising in models of optimal portfolio liquidation under market impact. Our existence and uniqueness of equilibrium result is based on a new existence of solutions result for FBSDE systems with singular drivers.

\subsection{Single player models of optimal portfolio liquidation}
Single-player portfolio liquidation models have been extensively analyzed in recent years. Their main characteristic is a singularity at the terminal time of the Hamilton-Jacobi-Bellmann equation. The majority of the optimal liquidation literature assumes that only absolutely continuous trading strategies are allowed; see e.g. \cite{AJK-2014,GHS-2013} and references therein. In such models the controlled state sequence follows a dynamics of the form
\[
	X_t=x-\int_0^t\xi_s\,ds,
\]
where $x > 0$ is the initial portfolio that a trader needs to unwind, and $\xi$ is the trading rate. The set of admissible controls is confined to those processes $\xi$ that satisfy almost surely the liquidation constraint
\[
	X_{T} = 0.
\]
It is typically assumed that the unaffected benchmark price process follows a one-dimensional Brownian motion $W$ (or some Brownian martingale) and that the trader's transaction price is given by
\[
	S_t=s_0+\int_0^t\sigma_s\,d W_s-\int_0^t\kappa_s\xi_s\,ds-\eta_t\xi_t
\]
where $\sigma$ is a (sufficiently regular) stochastic volatility process.
The integral term accounts for permanent price impact, i.e. the impact of past trades on current prices, while the term $\eta_t \xi_t$ accounts for the instantaneous impact that does not affect future transactions. The expected cost functional is typically of the linear-quadratic form
\[
	\bE\left[\int_0^T\!\!\!\Big(\kappa_s \xi_s X_s + \eta_s\xi_s^2+\lambda_sX_s^2  \Big)\,ds \right]
\]
where $\kappa, \eta$ and $\lambda$ are one-dimensional bounded adapted and non-negative processes. The process $\lambda$ describes the trader's degree of risk aversion or her belief about the volatility process; it penalizes slow liquidation. The process $\eta$ describes the degree of market illiquidity; it penalizes fast liquidation. The process $\kappa$ describes the impact of past trades on current transaction prices.

There are basically two approaches to overcome the challenges resulting from the terminal state constraint. The majority of the literature, including Ankirchner et al.~\cite{AJK-2014}, Graewe et al.~\cite{GHQ-2015}, Kruse and Popier~\cite{KP-2016} and Popier \cite{P-2006,Popier-2007} considers finite approximations of the singular terminal value, and then shows that the minimal solution to the value function with singular terminal condition can be obtained by a monotone convergence argument. A second approach, originally introduced in Graewe et al. \cite{GHS-2013} and further generalised in Graewe and Horst \cite{GH-2017} is to determine the precise asymptotic behaviour of a potential solution to the HJB equation at the terminal time, and to characterize the value function in terms of a PDE or BSDE with finite terminal value yet singular driver, for which the existence of a solution in a suitable space can be proved using standard fixed point arguments.

If the transactions are not directly observable, then it is natural to assume that the permanent impact is driven by the market’s expectation about the trader’s transactions as in \cite{basei:pham:17}, given the publicly observable information. It leads to a mean field control problem, which is not the focus of our paper.

\subsection{A MFG of optimal portfolio liquidation}

We consider a MFG of optimal portfolio liquidation among asymmetrically informed players. In order to introduce the game, we fix a probability space $(\Omega, \mathcal{G},\mathbb P)$ that carries independent standard Brownian motions $W^0, W^1, ..., W^N$ with $W^0$ in one dimension and $W^i$ in $m-1$ dimension and independent and identically distributed one-dimensional random variables $\cX^1, ...., \cX^N$ with law $\nu$ that are independent of the Brownian motions. The Brownian motion $W^0$ drives the unaffected benchmark price process. We assume that $W^0$ is observable by all agents. The Brownian motion $W^i$ is private information to player $i$ and determines that player's cost function. We may think of $W^i$ as measuring a player's individual degree of market impact and/or subjective belief about the price volatility. The random variables $\cX^1, ...., \cX^N$ specify the respective players' initial portfolios. We assume that each player observes the realization of her initial portfolio and knows the distribution of all the other initial portfolios.

Following \cite{CL-2016} we assume that the transaction price for each player $i=1, ..., N$ is given by
\[
    S^i_t=s^i_0+\int_0^t\sigma^i_s\,d W^0_s-\int_0^t\frac{\kappa^i_s}{N}\sum_{j=1}^N\xi^j_s\,ds-\eta^i_t\xi^i_t.
\]
In particular, the permanent price impact depends on the players' average trading rate.
Given her initial portfolio $\cX^i = x^i$ the optimization problem of player $i=1, ..., N$ is to minimize the cost functional
\begin{equation}\label{cost-player-i}
    J^{N,i}\left(\vec \xi\right)=\bE\left[\int_0^T \left(\frac{\kappa^i_t}{N}\sum_{j=1}^N\xi^j_tX^i_t+\eta^i_t(\xi^i_t)^2+\lambda^i_t (X^i_t)^2\right) dt \big | \cX^i = x^i\right]
\end{equation}
subject to the state dynamics
\begin{equation}\label{state-player-i}
\begin{split}
    dX^i_t &=-\xi^i_t\,dt, \\
    X^i_0 & =\cX^i, \quad X^i_T=0.
   \end{split}
\end{equation}
Here, $\vec{\xi}=(\xi^1,\cdots,\xi^N)$ is the vector of strategies of all the players. We assume that the one-dimensional cost coefficients $(\kappa^i, \eta^i, \lambda^i)$ have the same distribution across players and are adapted to the filtration
\begin{equation}\label{Fi}
	\bF^i:=(\mathcal F^i_t,0\leq t\leq T), \quad \mbox{with} \quad
	\mathcal F^i_t:=\sigma({\mathcal X^i},W^0_s,W^i_s,0\leq s\leq t).
\end{equation}

\begin{remark}
As pointed above, we assume that all players observe the Brownian motion $W^0$ that drives the unaffected benchmark price process; this motivates the individual information sets. Although we strongly believe that all our results carry over to the more general case where the agents observe only the price process itself, we prefer to work under the stronger assumption that $W^0$ is observable to simplify the analysis.
\end{remark}

Our game is different from the majority of the MFG literature in at least three respects. First, as in \cite{CL-2015, Gomes2013} the players interact through the impact of their strategies rather than states on the other players' payoff functions. Second, all players observe the common Brownian motion $W^0$ that drives the benchmark price process. Hence, ours is a MFG with common noise. While MFGs with common noise have been investigated before (see, e.g.~\cite{CDL-2016}) the nature of both the common and the idiosyncratic noise in our model is very different from the existing literature. Third, the individual state dynamics are subject to a terminal state constraint arising from the liquidation requirement.  MFGs with terminal state constraint have been considered before in the literature by means of so-called mean field (game) planning problems (MFGP) introduced by Lions in his lectures at Coll\`ege de France
(2009-2010). In these problems the terminal state constraint is given by a target density of the state at the terminal time. While our problem formally belongs to the literature on MFGP, see e.g. \cite{Achdou2012,Gomes2018,Porretta2014} and the references therein, ours seems to be the first paper that considers a MFG with strict terminal state constraint.

\subsubsection{The MFG}

In order to specify the resulting MFG, let $W^0$ and $W$ be independent Brownian motions of dimension $1$ and $m-1$, respectively\footnote{The same as $N$ player game, we may interpret $W^0$ as the common information to all players and $W$ as the private information to the representative player.}, and $\cX$ be an independent one-dimensional random variable with law $\nu$ defined on some probability space, again denoted $(\Omega,\mathcal G, \bP)$. Let  $\bF^0:=(\mathcal F^0_t,0\leq t\leq T)$ with $\mathcal F^0_t=\sigma(W^0_s,0\leq s\leq t)$ be the filtration generated by $W^0$ and let $\bF:=(\mathcal F_t,0\leq t\leq T)$ with $\mathcal F_t:=\sigma(\mathcal X,W^0_s,W_s,0\leq s\leq t).$ The MFG associated with the $N$-player game \eqref{cost-player-i} and \eqref{state-player-i} is then given by:
\begin{equation} \label{MFG-PVE}
\left\{\begin{array}{ll}
    1.\textrm{ fix a } \bF^0\textrm{ progressively measurable process }\mu~ (\textrm{in some suitable space});\\
    2.\textrm{ solve the corresponding constrained stochastic control problem}:\\
    \quad \inf_{\xi}\bE\left[\left.\int_0^T\left(\kappa_s \mu_sX_s+\eta_s\xi^2_s+\lambda_s X_s^2\right)\,ds\right|\mathcal X\right]\\
   \textrm{ subsect to } \\
    \quad dX_t=-\xi_t\,dt,X_0=\mathcal X~\textrm{ and }X_T=0;\\
    3.\textrm{ search for the fixed point }
   \mu_t=\bE[\xi^{*}_t|\cF^0_t],\textrm{ for }a.e.~ t\in[0,T]
\end{array}\right.
\end{equation}
where $\xi^{*}$ is the optimal strategy from 2 and the processes $(\kappa,\eta, \lambda)$ are adapted to the filtration $\bF$. We denote by $\mu^*$ a solution to the fixed point problem in Step 3. 

We apply the probabilistic method to solve the MFG with terminal constraint \eqref{MFG-PVE}. In a first step we show how the analysis of our MFG can be reduced to the analysis of a conditional mean-field type FBSDE. The forward component describes the optimal portfolio process; hence both its initial and terminal condition are known. The backward component describes the optimal trading rate; its terminal value is unknown. Making an affine ansatz, we show that the mean-field type FBSDE with unkown terminal condition can be replaced by a coupled FBSDE with known initial and terminal condition, yet singular driver. Proving the existence of a small time solution to this FBSDE by a fixed point argument is not hard. The challenge is to prove the existence of a global solution on the whole time interval. Under a weak interaction condition that has been used in the game theory literature before (see, e.g.~\cite{H-2005}) we prove the existence and uniqueness of a global solution by a generalization of the method of continuation established in \cite{hu:95,PW-1999} to linear-quadratic FBSDE systems with singular driver. Under the additional assumption that all players share the same cost structure, i.e.~that
\[
	\kappa^i = \kappa(\cX^i, W^0, W^i), \quad \eta^i = \eta(\cX^i, W^0, W^i), \quad 	\lambda^i = \lambda(\cX^i, W^0, W^i)
\]
for bounded measurable functions $\kappa, \eta, \lambda$ we prove that each player's best response to the mean-field equilibrium $\mu^*$ is of the form
\[
	\xi^{*,i} = \phi(\cX^i, W^0, W^i)
\]
for some function $\phi$ and that the resulting homogeneous action profile forms an $\epsilon$-equilibrium in the original $N$-player game.

The common information case where all the cost coefficients are measurable with respect to the common factor can be analysed in greater detail. When different players hold different initial portfolios, then the optimal portfolio processes are given as weighted averages of the players' initial portfolios and the differences of their own and the average initial portfolio. In this case, we show that if the average initial portfolio is positive and a player holds an above average initial portfolio, then her optimal portfolio process is always positive. If, however, a player holds a positive yet well below average initial portfolio, then it is optimal to quickly unwind the position, to then take a negative position and to buy the stock back by the end of the trading period. This is intuitive as players with negative portfolios benefit from the negative price trends generated by other players while the cost of unwinding a small portfolio is low. As such, our result suggests that traders with small portfolios act as liquidity providers in equilibrium even if their initial holds are positive.

The benchmark case of deterministic coefficients can be solved in closed form. For this case we show that when the strength of interaction $\kappa$ in \eqref{MFG-PVE} is large and all players share the same initial portfolio, the players initially trade very fast in equilibrium to avoid the negative drift generated by the mean field interaction. Our model thus provides a possible explanation for large price drops in markets with many strategically interacting homogenous investors. We also show that the deterministic case is equivalent to a single player model with suitably adjusted cost terms.

Under mild additional assumptions on the market impact parameters we further prove that the solution to the MFG can be approximated by the solutions to a sequence of MFGs where the liquidation constraint is replaced by an increasing penalization of open positions at the terminal time. The convergence result can be viewed as a consistency result for both, the unconstrained and the constrained problem.

The three papers closest to our model are Cardaliaguet and Lehalle \cite{CL-2016}, Carmona and Lacker \cite{CL-2015}, Huang, Jaimungal and Nourin \cite{HJN-2015}. In \cite{CL-2015}, the authors propose a specific portfolio liquidation model where each players portfolio is subject to exogenous fluctuations (customer flow) described by independent Brownian motions. As such, their model is much closer to a standard MFG than ours, but no liquidation constraint is possible in their framework. The papers \cite{CL-2016} and \cite{HJN-2015} consider mean field models parameterized by different preferences and with major-minor players, respectively. Again, no liquidation constraint is allowed. The model introduced in \cite{CL-2016} is extended to portfolios of correlated assets in \cite{lehalle:mouz:19} where the effect of trading flows on naive estimates of intraday volatility and correlations is analyzed.

The remainder of the paper is organized as follows. In Section \ref{sec:results} we state and prove our existence and uniqueness of solutions result for the MFG \eqref{MFG-PVE} and establish additional results on the equilibrium trading strategies and portfolio processes if all the players share the same information. In Section 3 we prove that the solution to the MFG yields an $\epsilon$-Nash equilibrium in the $N$-player game. In Section \ref{sec:convergence} we prove that the MFG with singular terminal condition can be approximated by MFGs that penalize open positions at the terminal time under additional assumptions on the market impact term.

\subsubsection{Notation and notational conventions}
Throughout, we adopt the convention that $C$ denotes a constant which may vary from line to line. Moreover, for a filtration $\bG$, $\prog(\bG)$ denotes the sigma-field of progressive subsets of $[0,T]\times \Omega$ and for $\mathbb I$, which could be a subset of $\mathbb R^n$, $n\geq 1$ or $\mathbb R\cup\{+\infty\}$, we consider the set of progressively measurable processes w.r.t. $\bG$:
$$\cP_\bG ([0,T]\times\Omega;\mathbb{I}) = \left\{u:[0,T]\times\Omega\rightarrow\mathbb{I}~\left| \ u \ \textrm{is } \prog(\bG)-\textrm{measurable}\right.\right\}.$$
We define the following subspaces of $\cP_\bG ([0,T]\times\Omega;\mathbb{I})$:
\begin{align*}
    L^{\infty}_{\bG}([0,T]\times\Omega;\mathbb{I})  =& \left\{ u \in \cP_\bG ([0,T]\times\Omega;\mathbb{I}) ; \ \|u\|:=\esssup_{t,\omega}|u(t,\omega)|<\infty \right\} \ ;  \\
       L^{p}_{\bG}([0,T]\times\Omega;\mathbb{I})  =& \left\{u \in \cP_\bG ([0,T]\times\Omega;\mathbb{I}); \ \bE \left( \int_0^T|u(t,\omega)|^2 dt\right)^{p/2}<\infty \right\};\\
       S^p_{\mathbb G}([0,T]\times\Omega;\mathbb I)=& \left\{u \in \mathcal P_{\mathbb G} ([0,T]\times\Omega;\mathbb{I}); \ \mathbb E \left( \sup_{0\leq t\leq T}|u(t,\omega)|^p \right)<\infty \right\}.
\end{align*}

Whenever the notation $T-$ appears in the definition of a function space we mean the set of all functions whose restriction satisfy the respective property on $[0,\tau]$ for any $\tau<T$, e.g., by $\psi\in L^2([0,T-]\times\Omega;\mathbb I)$, we mean $\psi\in L^2([0,\tau]\times\Omega;\mathbb I)$ for any $\tau<T$. For notational convenience, we put
\[
    D^2_\bG([0,T]\times \Omega;\mathbb I):=L^2_{\bG}([0,T]\times \Omega;\mathbb I) \cap S^2_{\bG}([0,T-]\times \Omega;\mathbb I).
\]
For a positive stochastic process $u \in L^{\infty}_{\bG}([0,T]\times\Omega;[0,\infty))$ we denote its upper and lower bound by {$u_{\max}$ and $u_{\min}$, respectively.

\section{The mean-field game}\label{sec:results}

In this section, we state and prove an existence and uniqueness of solutions result for the MFG \eqref{MFG-PVE}. The set of admissible controls for the representative player's liquidation problem is given by
\[
    \cA_\bF(\cX) := \left\{ \xi \in L^2_\bF([0,T]\times \Omega;\mathbb R), \ \int_{0}^T \xi_s \,ds = \cX \mbox{ a.s. }\right\}.
\]
For a given process $\mu\in L^2_{\bF^0}([0,T]\times\Omega;\bR)$, the corresponding cost and value functions are given by
$$J(\mathcal X,\xi;\mu) :=\bE \left[\left. \int_0^T\left( \kappa_sX_s\mu_s+\eta_s\xi^2_s+\lambda_s X_s^2\right)\,ds\right|\mathcal X\right],$$
and
\[
    V(\mathcal X;\mu)=\inf_{\xi\in\cA_{\bF}(\mathcal X)}J(\mathcal X,\xi;\mu),
\]
respectively. The Hamiltonian is
\[
   H(t,\xi,x,y;\mu)=-\xi y+\kappa_t\mu x+\eta_t \xi^2+\lambda_t x^2,
\]
and the stochastic maximum principle suggests that the solution to the optimization problem can be characterised in terms of the FBSDE
\begin{equation}\label{optimality-system-PVE}
\left\{\begin{aligned}
dX_t =&-\xi_t\,dt,\\
-dY_t =&\left(\kappa_t\mu_t+2\lambda_tX_t\right)\,dt-Z_t\,d\widetilde W_t,\\
X_0 =&\mathcal X\\
X_T = &0,
\end{aligned}
\right.
\end{equation}
where $\widetilde W=(W^0,W)$ is a $m$-dimensional Brownian motion. The process $Y$ is called the adjoint process to the controlled state process $X$. The liquidation constraint $X_T=0$ results in a singularity of the value function at liquidation time; see \cite{AJK-2014,GHS-2013}. As a result, the terminal condition for $Y$ cannot be determined a priori. In particular, the first equation holds on $[0,T]$ while the second equation holds on $[0,T)$. A standard approach yields the candidate optimal control
\begin{equation}\label{optimal-candidate-control-PVE}
    \xi^*_t=\frac{Y_t}{2\eta_t}.
\end{equation}
Taking the equilibrium condition into account suggests that the analysis of the MFG reduces to the analysis of the following conditional mean-field type FBSDE:
\begin{equation}\label{conditional-MF-FBSDE}
\left\{\begin{aligned}
dX_t =&-\dfrac{Y_t}{2\eta_t}\,dt,\\
-dY_t =&\left(\kappa_t\bE\left[\left.\frac{Y_t}{2\eta_t}\right|\cF^0_t\right]+2\lambda_tX_t\right)\,dt-Z_t\,d\widetilde{W}_t,\\
X_0 =&\mathcal X\\
X_T = &0.
\end{aligned}
\right.
\end{equation}
We establish the existence and uniqueness of a solution to the preceding FBSDE in the following space of weighted stochastic processes.

\begin{definition} \label{def:space_H_M}
For $l\in\mathbb R$, we introduce the space
\[
    \mathcal H_l := \{Y\in \cP_{\bF}([0,T]\times \Omega;\mathbb R\cup\{\infty\}): \ (T-.)^{-l} Y_\cdot \in S^2_\bF([0,T]\times \Omega;\mathbb R\cup\{\infty\}) \},
\]
which is endowed with the norm 
\[
	\|Y\|_{\mathcal H_l}:=\|Y\|_{l}:=\left(\bE\left[\sup_{0\leq t\leq T} \left|\frac{Y_t}{(T-t)^{l}}\right|^2\right]\right)^{\frac{1}{2}},
\]
and the space
\[
    \mathcal M_l:=\{Y\in \cP_{\bF}([0,T]\times \Omega;\mathbb R\cup\{\infty\}): \ (T-.)^{-l} Y_\cdot \in L^\infty_\bF([0,T]\times \Omega;\mathbb R\cup\{\infty\}) \},
\]
which is endowed with the norm
\[
    \|Y\|_{\mathcal M_l}:=\esssup_{(t,\omega)\in[0,T]\times\Omega}\frac{|Y_t|}{(T-t)^l}.
\]
\end{definition}
\begin{Fact} \label{fact_H_alpha}
The following facts are readily verified:
\begin{itemize}
\item For any $l$, $\mathcal H_{l}\subset \mathcal H_{-1+l}$ with $\|\cdot\|_{\mathcal H_{-1+l}}\leq T\|\cdot\|_{\mathcal H_l}$.
\item If $K \in \cH_l$, with $l > 0$, then $K_T = 0$ a.s.
\item If $K_1 \in \cM_{-1}$ and $K_2\in \cH_{l}$, then $K_1K_2 \in \cH_{-1+l}$.
\end{itemize}
The first two properties also hold for the space $\cM_l$.
\end{Fact}

We assume throughout that the cost coefficients are bounded and that the dependence of an individual player's cost function on the average action is weak enough. The weak interaction condition is consistent with the game theory literature on mean-field type games where some form of  moderate dependence condition is usually required to prove the existence and uniqueness of Nash equilibria; see \cite{H-2005} and references therein. The condition is also consistent with the monotonicity condition for FBSDE systems originally proposed by \cite{hu:95,PW-1999} and  
the generalization to mean-field type FBSDEs established in \cite{BensoussanYamZhang15}.  Specifically, we assume that the following condition is  satisfied.

\begin{ass}\label{ass-PVE}
$\ $
\begin{itemize}
\item[i)] The processes $\kappa$, $\lambda$, $1/\lambda$, $\eta$ and $1/\eta$ belong to $L^{\infty}_{\bF}([0,T]\times\Omega;[0,\infty))$ and $\mathcal X\in L^2(\Omega)$ is independent of $W$ and $W^0$.
\item[ii)] There exists a constant $\theta>0$ such that
\begin{equation} \label{ineq-PVE}
 \frac{\kappamax}{4\etamin}< \theta< 4 \frac{\lambdamin}{\kappamax}.
\end{equation}
\end{itemize}
\end{ass}

{{The following quantity will be important in our subsequent analysis:}}
\begin{equation} \label{alpha}
	\alpha:=\etamin/\etamax\in(0,1].
\end{equation}

We are now ready to state our first major result.

\begin{theorem}\label{main-result}
Under Assumption \ref{ass-PVE}, there exists a unique solution
\[
	(X,Y,Z)\in\mathcal H_\alpha \times L^2_{\bF}([0,T]\times\Omega;\mathbb R)\times  L^2_{\bF}([0,T-]\times\Omega;\mathbb R^m)
\]	
to the FBSDE \eqref{conditional-MF-FBSDE}. Moreover, the process $\xi^{*} = \dfrac{Y}{2\eta}$ is an optimal control for the representative player, $X^*=X$ is the optimal state process and the aggregation effect given by 
\begin{equation} \label{eq:consistency_condition}
	\mu^*_t=\bE\left[\left. \frac{Y_t}{2 \eta_t}\right|\mathcal F^0_t\right], \quad t \in [0,T)
\end{equation}
is the unique solution to the MFG \eqref{MFG-PVE}. Finally, the value function is given by
\begin{equation}\label{value-function-PVE}
    V(\mathcal X;\mu^*)=\frac{1}{2}A_{0} {\mathcal X}^2+\frac{1}{2}B_{0} \mathcal X + \frac{1}{2} \bE \left[ \int_0^{T} \kappa_sX^*_s\mu^*_s\,ds\bigg|\mathcal X\right].
\end{equation}

\end{theorem}
{Section \ref{existence-verification} is devoted to the proof of Theorem \ref{main-result} and Section \ref{sec-CVE} explores some particular cases. }

%%%%%%%%%%%%%%%%%%%%%%%%%%%%%%%%%%%%%%%%
%%%%%%%%%%%%%%%%%%%%%%%%%%%%%%%%%%%%%%%%
%%%%%%%%%%%%%%%%%%%%%%%%%%%%%%%%%%%%%%%%

\subsection{General existence and uniqueness of solutions}\label{existence-verification}

In this section we prove our existence and uniqueness of equilibrium result for the MFG \eqref{MFG-PVE}. Decoupling the FBSDE \eqref{conditional-MF-FBSDE} by $Y=AX+B$ yields the following system of Riccati type equations:
\begin{equation}\label{Riccati-system-PVE}
\left\{\begin{aligned}
    -dA_t =&\left(2\lambda_t-\frac{A^2_t}{2\eta_t}\right)\,dt-Z^A_t\,d\widetilde{W}_t,\\
    -dB_t =&\left(\kappa_t \bE\left[\left.\frac{1}{2\eta_t}\left(A_tX_t+B_t\right)\right|\cF^0_t\right]-\frac{A_tB_t}{2\eta_t}\right)\,dt-Z^B_t\,d\widetilde{W}_t,\\
    A_T =& \infty \\
    B_T =& 0.
\end{aligned}\right.
\end{equation}
The existence of a unique solution $A \in \cM_{-1}$ to the first equation is established in Lemma \ref{existence-asymptotic-expansion-A} in the appendix. Namely, there exists a unique process $(A,Z^A)$ such that $A  \in \cM_{-1}$, $Z^A \in L^2_{\bF}([0,T-]\times\Omega;\bR^m)$, the dynamics is given on any interval $[0,\tau]$, $\tau < T$ by the first equation of \eqref{Riccati-system-PVE} and $\displaystyle \lim_{t\to T} A_t = +\infty = A_T$. Moreover $A$ satisfies the a priori estimate \eqref{asymptotic-expansion-A} in Lemma \ref{existence-asymptotic-expansion-A} in the appendix, from which it also follows that
\begin{equation} \label{expA}
	\exp\left(-\int_r^s \frac{A_u}{2\eta_u} du \right) \leq  \left( \frac{T-s}{T-r} \right)^{\alpha}
\end{equation}
for any $0\leq r\leq s < T$, where $\alpha$ is given by \eqref{alpha}. Hence we need to solve the following FBSDE:
\begin{equation}\label{conditional-MF-FBSDE_2}
\left\{\begin{aligned}
dX_t =&-\frac{1}{2\eta_t} (A_t X_t + B_t)\,dt,\\
    -dB_t =&\left(\kappa_t \bE\left[\left.\frac{1}{2\eta_t}\left(A_tX_t+B_t\right)\right|\cF^0_t\right]-\frac{A_tB_t}{2\eta_t}\right)\,dt-Z^B_t\,d\widetilde{W}_t,\\
    X_0 =&\mathcal X \\
B_T =& 0.
\end{aligned}
\right.
\end{equation}
Our approach is based on an extension of the method of continuation that accounts for the singularity of the process $A$ at the terminal time and hence for the singularity in the driver of the FBSDE. We apply the method of continuation to the triple $(X,B,Y=AX+B)$ rather than the pair $(X,B)$, and search for solutions
\[
	(X,B,Y=AX+B)\in\mathcal H_\alpha \times\mathcal H_\gamma\times L^2_{\bF}([0,T]\times\Omega;\mathbb R),
\]	
where $\alpha$ was defined in (\ref{alpha}) and $\gamma$ is any constant $$0<\gamma<\alpha\wedge1/2.$$ Specifically, the method of continuation will be applied to the FBSDE
\begin{equation}\label{conditional-MF-FBSDE-3}
\left\{\begin{aligned}
dX_t =&-\frac{1}{2\eta_t} (A_t X_t + B_t)\,dt,\\
    -dB_t =&\left(\kappa_t \fp \bE\left[\left.\frac{1}{2\eta_t}\left(A_tX_t+B_t\right)\right|\cF^0_t\right]+f_t-\frac{A_tB_t}{2\eta_t}\right)\,dt-Z^B_t\,d\widetilde{W}_t,\\
dY_t=&\left(-2\lambda_tX_t-\kappa_t\fp \bE\left[\left.\frac{A_tX_t+B_t}{2\eta_t}\right|\mathcal F^0_t\right]-f_t\right)\,dt+Z^Y_t\,d\widetilde W_t,\\
X_0 =&\mathcal X \\
B_T =& 0,
\end{aligned}
\right.
\end{equation}
where $\fp\in[0,1]$, $f\in L^2_{\bF}([0,T]\times\Omega;\mathbb R)$. {We emphasise that the first two equations hold on $[0,T]$, while the third equation holds on $[0,T)$.}

In a first step, we provide an a priori estimate for the processes $Z^B$ and $Z^Y$.

\begin{lemma} \label{lem:BMO_martingale}
Assume that $f \in L^2_{\mathbb F}([0,T]\times\Omega;\mathbb R)$ and that
there exists a solution $(X,B,Y,Z^B,Z^Y)$ to \eqref{conditional-MF-FBSDE-3}  such that
\[
	(X,B,Y)\in \mathcal H_\alpha \times\mathcal H_\gamma \times {S^2_{\bF}([0,T-]\times \Omega,\bR)}. 		
\]	
Then
\[
	(Z^B,Z^Y)\in L^2_{\bF}([0,T]\times\Omega;\mathbb R^m)\times {L^2_{\bF}([0,T-]\times\Omega;\mathbb R^m)} 	
\]
	and there exists a constant $C > 0$ such that
$$\bE\left[\int_0^T|Z^B_t|^2\,dt\right] \leq C \left(\|B\|^2_\gamma + \|X\|^2_\alpha+\bE\left[\int_0^T|f_t|^2\,dt\right]\right)$$
and such that for each $\tau <T$
\begin{equation*}
\begin{split}
    \bE\left[\int_0^{\tau}|Z_s^Y|^2\,ds\right]
    \leq C\left(\bE\left[\sup_{0\leq t\leq\tau}|Y_t|^2\right]+\|X\|^2_{\alpha}+\|B\|^2_\gamma+\bE\left[\int_0^{{T}} |f_t|^2\,dt\right]\right).
\end{split}
\end{equation*}
In particular, $\int_0^\cdot Z^B_s\,d\widetilde W_s$ is a true martingale on $[0,T]$ and $\int_0^\cdot Z^Y_s\,d\widetilde W_s$ is a true martingale on $[0,\tau]$, for each $\tau<T$.
\end{lemma}
\begin{proof}
Since $A\in\mathcal M_{-1}$ and $\etamin>0$ there exists a constant $C>0$ that is independent of $s \in [0,T]$ such that
\begin{equation*}
	\left| \frac{A_s B_s}{2\eta_s}- \kappa_s \bE\left[\left.\frac{A_s X_s+B_s}{2\eta_s} \right| \mathcal F^0_s \right]-f_s\right|
	\leq C\left[ \frac{|B_s|}{T-s} + \bE\left(\left.\frac{|X_s|}{T-s}+|B_s|\right|\mathcal F^0_s\right)+|f_s|\right] .
\end{equation*}
Let us notice that
$$ 	\int_t^T Z^B_s d\widetilde W_s
	 =  -B_t -\int_t^T \left\{\frac{A_s B_s}{2\eta_s}-\kappa_s \fp \bE \left[\left.\frac{A_s X_s+B_s}{2\eta_s} \right| \mathcal F^0_s \right]-f_s \right\}\,ds.$$
Since $(X,B)\in\mathcal H_{\alpha}\times\mathcal H_{\gamma}$, this implies
\begin{equation*}
    \begin{split}
 	\left| \int_t^T Z^B_s d\widetilde W_s \right|
		 \leq &~ C\sup_{0\leq t\leq T}\frac{|B_t|}{(T-t)^\gamma}+C\sup_{0\leq s\leq T}\bE\left[\left.\sup_{0\leq t\leq T}\frac{|X_t|}{(T-		t)^\alpha}\right|\mathcal F^0_s\right]\\
	&~ +C\sup_{0\leq s\leq T}\bE\left[\left.\sup_{0\leq t\leq T}\frac{|B_t|}{(T-t)^\gamma}\right|\mathcal F^0_s\right]+\int_0^T|f_t|\,dt.
    \end{split}
\end{equation*}
Thus, by Doob's maximal inequality,
\begin{equation*}
    \begin{split}
    \bE\left[\sup_{0\leq t\leq T}\left| \int_t^T Z^B_s d\widetilde W_s \right|^2\right]
 \leq  C \left(\|B\|^2_{\gamma}+\|X\|^2_{\alpha}+\bE\left[\int_0^T|f_t|^2\,dt\right]\right) .
    \end{split}
\end{equation*}
Similarly, for each $0<\tau<T$,
\begin{equation*}
\begin{split}
    \bE\left[\sup_{0\leq t\leq\tau}\left|\int_t^{\tau }Z_s^Y\,d\widetilde W_s\right|^2\right]
    \leq C\left(\bE\left[\sup_{0\leq t\leq\tau}|Y_t|^2\right]+\|X\|^2_{\alpha}+\|B\|^2_\gamma+\bE\left[\int_0^{T} |f_t|^2\,dt\right]\right)<\infty.
\end{split}
\end{equation*}
\end{proof}

In a second step, we now prove an existence of solutions result for the FBSDE \eqref{conditional-MF-FBSDE-3} {with $\fp=0$.}

\begin{lemma} \label{induction-basis-eta}
For $\fp=0$ there exists for every given data $f\in L^2_{\bF}([0,T]\times\Omega;\mathbb R)$ a unique solution $(X,B,Y,Z^B,Z^Y)\in\mathcal H_\alpha \times\mathcal H_\gamma\times D^2_{\bF}([0,T]\times \Omega;\mathbb R){\times L^2_{\mathbb F}([0,T]\times\Omega;\mathbb R^m)\times L^2_{\mathbb F}([0,T-]\times\Omega;\mathbb R^m)}$ to~\eqref{conditional-MF-FBSDE-3}. It is given by %\textcolor{blue}{for any $t\in [0,T]$
\begin{equation*}
\left\{\begin{aligned}
    B_t=&~\bE\left[\left.\int_t^Tf_se^{-\int_t^s (2\eta_r)^{-1}A_r\,dr}\,ds\right|\mathcal F_t\right], \quad t \in [0,T]\\
	X_t=&~\mathcal Xe^{-\int_0^t (2\eta_r)^{-1}A_r\,dr}-\int_0^t(2\eta_s)^{-1}B_se^{-\int_s^t (2\eta_r)^{-1} A_r\,dr}\,ds, \quad t \in [0,T] \\
	Y_t = & A_t X_t + B_t, \quad t \in [0,T),
\end{aligned}\right.
\end{equation*}
and {$Z^B \in L^2_{\bF}([0,T]\times \Omega;\mathbb R^m)$ and $Z^Y\in  L^2_{\bF}([0,T-]\times \Omega;\mathbb R^m)$} are given by the martingale representation theorem.
\end{lemma}
\begin{proof}
For $\fp=0$ the process $X$ solves a linear ODE and the pair $(B,Z^B)$ solves a linear BSDE. Hence, the explicit representations follow from the respective solution formulas. It remains to establish the desired integration properties. To this end, let us recall that $A$ has positive values. Thus we first apply H\"older's inequality in order to obtain,
\begin{equation*}
\begin{split}
    \frac{|B_t|}{(T-t)^\gamma}\leq&~ \frac{1}{(T-t)^\gamma}\bE\left[\left.\int_t^T|f_s|\,ds\right|\mathcal F_t\right]
    \le \left(\bE\left[\left.\int_t^T |f_s|^{\frac{1}{1-\gamma}}\,ds\right|\mathcal F_t\right]\right)^{1-\gamma} < \infty.
\end{split}
\end{equation*}
Using Doob's maximal inequality, Jensen's inequality and the fact that $\gamma < \frac{1}{2}$ we conclude that,
\begin{equation*}
\begin{split}
     \bE\left[\sup_{0\leq t\leq T}\left|\frac{B_t}{(T-t)^\gamma}\right|^2\right]\leq&~ \bE\left[\sup_{0\leq t\leq T}\left(\bE\left[\int_0^T |f_s|^{\frac{1}{1-\gamma}}\,ds\bigg|\mathcal F_t\right]\right)^{2(1-\gamma)}\right]
     \leq C\bE\left[\int_0^T|f_s|^2\,ds\right].
\end{split}
\end{equation*}

From (\ref{expA}), the solution formula for $X$ and using that $\gamma < \alpha$ we obtain that $X\in\mathcal H_\alpha$ because
\begin{equation*}
\begin{split}
 |X_t| \leq &~ \frac{|\cX|(T-t)^\alpha}{T^\alpha}+C\int_0^t |B_s| \left(\frac{T-t}{T-s}\right)^\alpha\,ds \\
 \leq &~  \frac{|\cX|(T-t)^{\alpha}}{T^\alpha}+C \left(  \sup_{0\leq s \leq T} \frac{|B_s|}{(T-s)^\gamma} \right) \left( \int_0^t \left(T-s\right)^{\gamma-\alpha} \,ds \right) (T-t)^\alpha \\
 \leq &~  (T-t)^{\alpha} \left\{ \frac{|\cX|}{T^\alpha}+\frac{CT^{1+\gamma-\alpha}}{1+\gamma-\alpha} \left(  \sup_{0\leq s \leq T} \frac{|B_s|}{(T-s)^\gamma} \right) \right\}.
\end{split}
\end{equation*}
In view of \eqref{asymptotic-expansion-A} and the previously established properties of $X$ and $B$ we have $Y \in {S^2_{\bF}([0,T-]\times \Omega; \bR)}$ with
\begin{equation} \label{estY}
\bE \left[\sup_{t\in [0,\tau]} Y^2_t \right] \leq \frac{C}{(T-\tau)^{2(1-\alpha)}} \| X\|_{\alpha}^2 + (T-\tau)^{2\gamma} \|B\|^2_\gamma.
\end{equation}
For any $\epsilon>0$, integration by part implies that
\begin{equation*}
\begin{split}
    X_{T-\epsilon}Y_{T-\epsilon}{-X_0Y_0}=&~\int_0^{T-\epsilon}X_t\,dY_t+\int_0^{T-\epsilon}Y_t\,dX_t\\
    =&~-\int_0^{T-\epsilon}X_t(2\lambda_tX_t+f_t)\,dt-\int_0^{T-\epsilon}\frac{Y_t^2}{2\eta_t}\,dt+\int_0^{T-\epsilon}X_tZ^Y_t\,d\widetilde W_t.
\end{split}
\end{equation*}
The positivity of the process $A$ along with the definition of the process $Y$ yields
$
    X_{T-\epsilon}Y_{T-\epsilon}
    \geq X_{T-\epsilon}B_{T-\epsilon}.
$
Thus, taking expectations on both sides of the above equation, letting $\epsilon\rightarrow 0$ and using ${X_{T}=B_T=0}$ yields
\begin{equation*}
\begin{split}
    {-\bE\left[X_0Y_0\right]}\leq-\bE\left[\int_0^T2\lambda_tX_t^2\,dt\right]-\bE\left[\int_0^TX_tf_t\,dt\right]-\bE\left[\int_0^T\frac{Y_t^2}{2\eta_t}\,dt\right].
\end{split}
\end{equation*}
Together with the inequality \eqref{estY} for $\tau = 0$ this shows that
\[
    \bE\left[\int_0^TY_t^2\,dt\right]\leq C\bE\left[\int_0^TX_t^2\,dt\right]+C\bE\left[\int_0^Tf_t^2\,dt\right] + C\|X\|_\alpha^2 + C \|B\|_\gamma^2<\infty.
\]
\end{proof}

In a third step we now establish the continuation result for the FBSDE \eqref{conditional-MF-FBSDE-3} from which we shall then deduce the existence of a unique global solution to our original MFG.

\begin{lemma} \label{induction-step-eta}
If for some $\fp \in[0,1]$ the FBSDE \eqref{conditional-MF-FBSDE-3} is for every data $f\in L^2_{\bF}([0,T]\times\Omega;\mathbb R)$ uniquely solvable in $\cH_\alpha \times\cH_\gamma\times {D^2_{\bF}([0,T]\times \Omega;\mathbb R)}\times L^2_{\bF}([0,T]\times \Omega;\mathbb R^m)\times L^2_{\bF}([0,T-]\times \Omega;\mathbb R^m)$,  then this holds also for $\fp+\fd$ with $\fd>0$ small enough (independent of $\fp$ and $f$).
\end{lemma}
\begin{proof}
Let us fix $\fd>0$, $Y\in L^2_{\bF}([0,T]\times \Omega;\mathbb R)$ and $f\in L^2_{\bF}([0,T]\times \Omega;\mathbb R)$ and consider the following system:
\begin{equation}\label{conditional-MF-FBSDE-4}
\left\{\begin{aligned}
d\widetilde X_t =&~-\frac{1}{2\eta_t} (A_t\widetilde X_t +\widetilde B_t)\,dt,\\
    -d\widetilde B_t =&~\left(\kappa_t \fp \bE\left[\left.\frac{1}{2\eta_t}\left(A_t\widetilde X_t+\widetilde B_t\right)\right|\cF^0_t\right]+\kappa_t \fd \bE\left[\left.\frac{Y_t}{2\eta_t}\right|\cF^0_t\right]+f_t-\frac{A_t\widetilde B_t}{2\eta_t}\right)\,dt-Z^{\widetilde B}_t\,d\widetilde{W}_t,\\
d\widetilde Y_t=&~\left(-2\lambda_t\widetilde X_t-\kappa_t\fp \bE\left[\left.\frac{A_t\widetilde X_t+\widetilde B_t}{2\eta_t}\right|\mathcal F^0_t\right]-\kappa_t\fd \bE\left[\left.\frac{Y_t}{2\eta_t}\right|\mathcal F^0_t\right]-f_t\right)\,dt+Z^{\widetilde Y}_t\,d\widetilde W_t,\\
\widetilde X_0 =&~\mathcal X \\
\widetilde B_T =&~ 0.
\end{aligned}
\right.
\end{equation}
Then
\[
	f(Y):=\kappa\fd \bE\left[\left.\frac{Y}{2\eta}\right|\mathcal F^0\right]+f\in L^2_{\bF}([0,T]\times \Omega;\mathbb R).
\]	
	Thus, by assumption there exists a unique solution $$(\widetilde X,\widetilde B,\widetilde Y,Z^{\widetilde B},Z^{\widetilde Y})\in \mathcal H_\alpha\times\mathcal H_\gamma\times {D^2_{\bF}([0,T]\times \Omega;\mathbb R)}\times L^2_{\bF}([0,T]\times \Omega;\mathbb R^m)\times L^2_{\bF}([0,T-]\times \Omega;\mathbb R^m)$$ to \eqref{conditional-MF-FBSDE-4}, { and $\widetilde Y=A\widetilde X+\widetilde B$.}
This defines a mapping $Y \mapsto (\widetilde X,\widetilde B,\widetilde Y)$ from $L^2_{\bF}([0,T]\times \Omega;\mathbb R)$ to $\cH_\alpha\times\cH_\gamma\times L^2_{\bF}([0,T]\times \Omega;\mathbb R)$ and hence also a mapping $(X,B,Y) \mapsto (\widetilde X,\widetilde B,\widetilde Y)$ on $\cH_\alpha\times\cH_\gamma\times L^2_{\bF}([0,T]\times \Omega;\mathbb R)$. In what follows we prove that this second mapping is a contraction for some $\fd >0$. For the unique fixed point the system (\ref{conditional-MF-FBSDE-4}) reduces to the system (\ref{conditional-MF-FBSDE-3}) with $\fp$ replaced by $\fp + \fd$. This then yields the desired result.

In order to establish the contraction property, we denote for two processes $Y,Y' \in L^2_{\bF}([0,T]\times \Omega;\mathbb R)$ by $(\widetilde X, \widetilde B, \widetilde Y)$ and $(\widetilde X', \widetilde B', \widetilde Y')$ the corresponding processes defined by \eqref{conditional-MF-FBSDE-4} and put
\[
 \widetilde \xi_t = \dfrac{\widetilde Y_t}{2\eta_t}, \quad \widetilde \xi'_t = \dfrac{\widetilde Y'_t}{2\eta_t}, \quad
\widetilde \mu_t =
 \bE \left[\left.\widetilde \xi_t \right| \cF^0_t\right], \quad
 \widetilde \mu'_t = \bE \left[\left.\widetilde \xi'_t \right| \cF^0_t\right].
\]

For any $\epsilon > 0$ integration by part yields that
\begin{equation*}
\begin{split}
 (\widetilde X^{'}_{T-\epsilon} - \widetilde X_{T-\epsilon}) \widetilde Y_{T-\epsilon} = &~\int_0^{T-\epsilon} (\widetilde X^{'}_{s} - \widetilde X_{s}) \,d\widetilde Y_s+\int_0^{T-\epsilon} \widetilde Y_s\,d (\widetilde X^{'}_{s} - \widetilde X_{s})\\
 = &~ - \int_0^{T-\epsilon} (\widetilde X^{'}_{s} - \widetilde X_{s})(\fp \kappa_s \widetilde \mu_s+2\lambda_s \widetilde X_s)\,ds-\int_0^{T-\epsilon} \widetilde Y_s(\widetilde \xi'_s- \widetilde \xi_s)\,ds\\
 &~- \int_0^{T-\epsilon} (\widetilde X^{'}_{s} - \widetilde X_{s})f(Y_s)\,ds  +\int_0^{T-\epsilon} (\widetilde X^{'}_{s} - \widetilde X_{s})  Z^{\widetilde Y}_s\,d\widetilde W_s
\end{split}
\end{equation*}
and
\begin{equation*}
\begin{split}
 (\widetilde X_{T-\epsilon} - \widetilde X^{'}_{T-\epsilon}) \widetilde Y^{'}_{T-\epsilon}  =&~  - \int_0^{T-\epsilon} (\widetilde X_{s} - \widetilde X^{'}_{s})(\fp \kappa_s \widetilde \mu^{'}_s+2\lambda_s \widetilde X^{'}_s)\,ds-\int_0^{T-\epsilon} \widetilde Y^{'}_s(\widetilde \xi_s- \widetilde \xi'_s)\,ds\\
 &~ - \int_0^{T-\epsilon} (\widetilde X_{s} - \widetilde X^{'}_{s})f(Y^{'}_s)\,ds+\int_0^{T-\epsilon} (\widetilde X_{s} - \widetilde X^{'}_{s}) Z^{\widetilde Y^{'}}_s\,d\widetilde W_s.
\end{split}
\end{equation*}
Taking the sum of these two equations and using that
\begin{equation*}
    \begin{split}
    (\widetilde X_{T-\epsilon} - \widetilde X^{'}_{T-\epsilon}) (\widetilde Y^{'}_{T-\epsilon} - \widetilde Y_{T-\epsilon})
    =&~ -A_{T-\epsilon}(\widetilde X_{T-\epsilon} - \widetilde X^{'}_{T-\epsilon})^2 - (\widetilde X_{T-\epsilon} - \widetilde X^{'}_{T-\epsilon})(\widetilde B_{T-\epsilon} - \widetilde B^{'}_{T-\epsilon}) \\
    \leq &~ - (\widetilde X_{T-\epsilon} - \widetilde X^{'}_{T-\epsilon})(\widetilde B_{T-\epsilon} - \widetilde B^{'}_{T-\epsilon})
    \end{split}
\end{equation*}
yields
\begin{equation*}
    \begin{split}
        &~2  \int_0^{T-\epsilon} \eta_s (\widetilde \xi'_s- \widetilde \xi_s)^2 \,ds + 2 \int_0^{T-\epsilon} \lambda_s  (\widetilde X^{'}_{s} - \widetilde X_{s})^2\, ds \\
        &~+ \int_0^{T-\epsilon} (\widetilde X^{'}_{s} - \widetilde X_{s})( f(Y'_s) -f(Y_s) )\,ds +\int_0^{T-\epsilon} (\widetilde X_{s} - \widetilde X^{'}_{s}) (\widetilde Z^{Y^{'}}_s - \widetilde Z^{Y}_s) \,d\widetilde W_s \\
        \leq&~ - (\widetilde X_{T-\epsilon} - \widetilde X^{'}_{T-\epsilon})(\widetilde B_{T-\epsilon} - \widetilde B^{'}_{T-\epsilon}) +  \int_0^{T-\epsilon}\left[\fp \kappa_s (\widetilde \mu_s - \widetilde \mu'_s)(\widetilde X'_s - \widetilde X_s) \right]\, ds .
    \end{split}
\end{equation*}
Taking expectations on both sides drops the martingale part. Then we can pass to the limit as $\epsilon \to 0$ to drop the term $(\widetilde X_{T-\epsilon} - \widetilde X^{'}_{T-\epsilon})(\widetilde B_{T-\epsilon} - \widetilde B^{'}_{T-\epsilon})$ because $\widetilde X$, $\widetilde X'\in\mathcal H_\alpha$ and $\widetilde B$, $\widetilde B'\in\cH_\gamma$. Furthermore, since $2|ab| \leq \theta a^2 +  b^2/\theta$ for any $\theta >0$, we obtain:
\begin{equation*}
    \begin{split}
        &\bE\left[  \int_0^{T}\left|\fp \kappa_s (\widetilde \mu_s - \widetilde \mu'_s)(\widetilde X'_s - \widetilde X_s) \right|\, ds\right]  \leq \frac{\fp \kappamax }{2\theta} \bE\left[  \int_0^{T}\left| \widetilde \mu_s - \widetilde \mu'_s\right|^2 \, ds\right] + \frac{\fp \kappamax  \theta}{2} \bE\left[  \int_0^{T}\left|\widetilde X'_s - \widetilde X_s \right|^2 \, ds\right] \\
        &\quad  \leq \frac{ \kappamax }{2\theta} \bE\left[  \int_0^{T}\left| \widetilde \xi_s - \widetilde \xi'_s\right|^2 \, ds\right] + \frac{ \kappamax  \theta}{2} \bE\left[  \int_0^{T}\left|\widetilde X'_s - \widetilde X_s \right|^2 \, ds\right],
    \end{split}
\end{equation*}
and
\begin{equation*}
    \begin{split}
        &\bE\left[  \int_0^{T}\left|(\widetilde X^{'}_{s} - \widetilde X_{s})( f(Y'_s) -f(Y_s) ) \right|\, ds\right] \leq \fd \bE\left[  \int_0^{T}\kappa_s \left|\widetilde X^{'}_{s} - \widetilde X_{s}\right| \bE\left[\left.\frac{|Y'_s - Y_s|}{2\eta_s}\right|\mathcal F^0_s \right] \, ds\right]  \\
        & \quad \leq \fd \frac{\kappamax }{2\etamin} \bE\left[  \int_0^{T}\left|\widetilde X^{'}_{s} - \widetilde X_{s}\right| \bE\left[\left. |Y'_s - Y_s|\right|\mathcal F^0_s \right] \, ds\right] \\
        &\quad \leq  \fd \frac{\kappamax }{4\etamin} \bE\left[  \int_0^{T}\left(\widetilde X^{'}_{s} - \widetilde X_{s}\right)^2  \, ds\right] +  \fd \frac{\kappamax }{4\etamin} \bE\left[  \int_0^{T} \bE\left[\left. (Y'_s - Y_s)^2 \right|\mathcal F^0_s \right]  \, ds\right].
    \end{split}
\end{equation*}
All these inequalities imply that 
for any $\theta > 0$
\begin{equation*}
    \begin{split}
        &~\left(2\etamin-\frac{\kappamax}{2\theta}\right) \bE\left[  \int_0^{T}  (\widetilde \xi'_s- \widetilde \xi_s)^2\, ds\right] + \left(2\lambdamin-\frac{\kappamax}{2}\theta\right) \bE\left[ \int_0^{T}   (\widetilde X^{'}_{s} - \widetilde X_{s})^2 \,ds\right] \\
        \leq&~\frac{\kappamax}{4\etamin}\fd \bE\left[   \int_0^{T}(\widetilde X'_s - \widetilde X_s)^2 \,ds\right] +\frac{\kappamax}{4\etamin} \fd \bE\left[ \int_0^T  (Y'_s - Y_s)^2\,ds\right].
    \end{split}
\end{equation*}
In view of Assumption \ref{ass-PVE} we can choose a $\theta>0$ such that
\[
    2\etamin-\frac{\kappamax}{2\theta}>0,\qquad 2\lambdamin-\frac{\kappamax\theta}{2}>0,
\]
which implies that there exists a constant $C$ depending only on the coefficients $\kappa$, $\lambda$ and {$\eta$}, such that
\begin{equation*}
    \begin{split}
        &~\bE\left[ \int_0^{T}  (\widetilde \xi'_s- \widetilde \xi_s)^2 \,ds\right] +  \bE\left[ \int_0^{T}   (\widetilde X^{'}_{s} - \widetilde X_{s})^2 \,ds\right] \\
        \leq&~ C\fd\bE\left[   \int_0^{T}(\widetilde X'_s - \widetilde X_s)^2 \,ds\right] +C \fd \bE\left[ \int_0^T  (Y'_s - Y_s)^2\,ds\right].
    \end{split}
\end{equation*}
Thus, when $\fd$ is small enough,
\[
    \bE\left[\int_0^T|\widetilde Y_t-\widetilde Y'_t|^2\,dt\right]\leq \mathfrak{a} \bE\left[\int_0^T|Y_t-Y'_t|^2\,dt\right]
\]
for some $\mathfrak{a} < 1$. We notice that the bound on $\fd$ only depends on $\kappa$, $\eta$ and $\lambda$.

Now using the definition of $\widetilde \xi$ and $\widetilde \xi'$ the solution formula for linear BSDEs yields
\begin{equation*}
|\widetilde B_t-\widetilde B_t'| \leq \kappamax \bE \left[ \int_t^T\left\{\fp\bE\left[|\widetilde \xi_s- \widetilde \xi'_s| \bigg| \cF^0_s \right]+\fd \bE\left[ \frac{|Y_s-  Y'_s|}{2\eta_s}\bigg| \cF^0_s\right]\right\}ds \bigg| \cF_t \right].
\end{equation*}
Thus
\begin{equation*}
    \begin{split}
|\widetilde B_t-\widetilde B_t'|  \leq&~ C(T-t)^\gamma \bE \left[\left. \int_t^T\bE\left[\left.|\widetilde \xi_s- \widetilde \xi'_s|^{\frac{1}{1-\gamma}} \right| \cF^0_s \right]  ds\right| \cF_t \right]^{1-\gamma}\\
&~+
C\fd (T-t)^\gamma \bE \left[\left. \int_t^T \bE\left[\left. |Y_s-  Y'_s|^{\frac{1}{1-\gamma}}  \right| \cF^0_s\right] ds \right| \cF_t \right]^{1-\gamma}.
    \end{split}
\end{equation*}
Since $2\gamma < 1$, Doob's maximal inequality along with the previously established $L^2$ bounds yields
$$ \bE \left[ \sup_{t\in [0,T]} \frac{|\widetilde B_t-\widetilde B_t'|^2}{(T-t)^{2\gamma}} \right] \leq  C\bE \left[ \int_0^T|\widetilde \xi_s- \widetilde \xi'_s|^2  ds \right] +C\fd^2\bE \left[ \int_0^T|Y_s-  Y'_s|^2 ds \right].$$
Now using the dynamics of $\widetilde X$ and $\widetilde X'$ we obtain
\begin{equation*}
    \begin{split}
|\widetilde X_t-\widetilde X_t'|=&~ \left|\int_0^t -\{\fp(2\eta_s)^{-1}(\widetilde B_s-\widetilde B_s')\}e^{-\int_s^t (2\eta_r)^{-1}A_r\,dr}\,ds\right| \\
	\leq&~ C\int_0^t \{|\widetilde B_s-\widetilde B_s'|\}\left(\frac{T-t}{T-s}\right)^\alpha\, ds\\
	\leq&~ C\frac{T^{1+\gamma-\alpha}}{1+\gamma-\alpha} (T-t)^\alpha\sup_{0\leq s\leq T} \frac{|\widetilde B_s-\widetilde B'_s|}{(T-s)^\gamma}.  \end{split}
\end{equation*}
Hence this leads to
\[
    \bE \left[ \sup_{t\in [0,T]} \frac{|\widetilde X_t-\widetilde X_t'|^2}{(T-t)^{2\alpha}} \right] \leq C\| \widetilde B_s-\widetilde B_s'\|^2_{\gamma}.
\]

To summarize, we obtain a constant $\fd$ such that $(X,B,Y) \to (\widetilde X, \widetilde B,\widetilde Y)$ is a contraction in $\cH_\alpha \times \cH_\gamma \times L^2_{\bF}([0,T]\times \Omega;\mathbb R)$. Since $\widetilde Y = A \widetilde X + \widetilde B$, $\widetilde Y \in D^2_\bF([0,T]\times\Omega;\bR)$ and using Lemma \ref{lem:BMO_martingale}, we see that the following system admits a unique solution $(\widetilde X,\widetilde B,\widetilde Y,Z^{\widetilde B},Z^{\widetilde Y})\in \mathcal H_\alpha \times\mathcal H_\gamma\times D^2_{\bF}([0,T]\times \Omega;\mathbb R)\times L^2_{\bF}([0,T]\times \Omega;\mathbb R^m)\times L^2_{\bF}([0,T-]\times \Omega;\mathbb R^m) $:
\begin{equation*}
\left\{\begin{aligned}
d\widetilde X_t =&~-\frac{1}{2\eta_t} (A_t\widetilde X_t +\widetilde B_t)\,dt,\\
    -d\widetilde B_t =&~\left(\kappa_t \fp \bE\left[\left.\frac{1}{2\eta_t}\left(A_t\widetilde X_t+\widetilde B_t\right)\right|\cF^0_t\right]+\kappa_t \fd \bE\left[\left.\frac{\widetilde Y_t}{2\eta_t}\right|\cF^0_t\right]+f_t-\frac{A_t\widetilde B_t}{2\eta_t}\right)\,dt-Z^{\widetilde B}_t\,d\widetilde{W}_t,\\
d\widetilde Y_t=&~\left(-2\lambda_t\widetilde X_t-\kappa_t\fp \bE\left[\left.\frac{A_t\widetilde X_t+\widetilde B_t}{2\eta_t}\right|\mathcal F^0_t\right]-\kappa_t\fd \bE\left[\left.\frac{\widetilde Y_t}{2\eta_t}\right|\mathcal F^0_t\right]-f_t\right)\,dt+Z^{\widetilde Y}_t\,d\widetilde W_t,\\
\widetilde X_0 =&~\mathcal X \\
\widetilde B_T =&~ 0.
\end{aligned}
\right.
\end{equation*}
Using again the relation $\widetilde Y=A\widetilde X+\widetilde B$, the above system is equivalent to \eqref{conditional-MF-FBSDE-3} with $\fp$ replaced by $\fp+\fd$. This proves the assertion.
\end{proof}

Using Lemmata \ref{lem:BMO_martingale}, \ref{induction-basis-eta} and \ref{induction-step-eta} and by induction on $\fp$, we obtain the following result.

\begin{proposition}\label{existence_MF_FBSDE}
There exists a unique solution $(X,B,Y,Z^{B},Z^{Y}) \in \cH_\alpha \times \cH_\gamma \times D^2_{\bF}([0,T]\times\Omega;\bR)\times L^2_{\bF}([0,T]\times \Omega;\mathbb R^m)\times L^2_{\bF}([0,T-]\times \Omega;\mathbb R^m)$ to the FBSDEs \eqref{conditional-MF-FBSDE} and \eqref{conditional-MF-FBSDE_2}. Moreover, there exists a constant $C>0$ depending on $\eta$, $\lambda$, $\kappa$, $T$ and $\|\mathcal X\|_{L^2}$, such that
\[
   \|X\|_{\cH_\alpha} + \|B\|_{\cH_\gamma} + \bE\left[\int_0^T|Y_t|^2\,dt\right] \leq C.
\]
\end{proposition}
From the equations \eqref{conditional-MF-FBSDE_2}, \eqref{optimal-candidate-control-PVE} and recalling $Y=AX+B$, where $(X,Y,B)$ is from Proposition \ref{existence_MF_FBSDE},
we obtain the following candidates of the optimal portfolio process $X^*$ and the optimal trading strategy $\xi^*$ for the representative player:
\begin{equation} \label{eq:optimal_candidates}
\begin{split}
    X^{*}_t & =X_t = \mathcal Xe^{-\int_0^t\frac{A_r}{2\eta_r}\,dr}-\int_0^t\frac{B_s}{2\eta_s}e^{-\int_s^t\frac{A_r}{2\eta_r}\,dr}\,ds,\\
    \xi^{*}_t & =\frac{Y_t}{2\eta_t} = \frac{A_t X_t+B_t}{2\eta_t}= \mathcal Xe^{-\int_0^t\frac{A_r}{2\eta_r}\,dr}\frac{A_t}{2\eta_t}+\frac{B_t}{2\eta_t}-\frac{A_t}{2\eta_t}\int_0^t\frac{B_s}{2\eta_s}e^{-\int_s^t\frac{A_r}{2\eta_r}\,dr}\,ds.
\end{split}
\end{equation}

By construction, $X^*_T = 0$ and hence $\xi^*$ is an admissible liquidation strategy.
The following proposition shows that it is indeed the optimal liquidation strategy and that its conditional expectation defines the desired equilibrium for our MFG. In particular, it proves Theorem \ref{main-result}.

\begin{proposition}\label{verification-PVE}
The process $\xi^{*}$ given by \eqref{eq:optimal_candidates} or equivalently by \eqref{optimal-candidate-control-PVE} is an optimal control for the representative player, $X^*$ is the related optimal state process, and the aggregation effect given by $\mu^*:=\bE[\xi^{*}| \cF^0]$ is the solution to the MFG \eqref{MFG-PVE}. Moreover, the value function is given by \eqref{value-function-PVE}. 
\end{proposition}
\begin{proof}
Let $(X,B,Y)$ be the solution given by Proposition \ref{existence_MF_FBSDE}.
For any $\xi \in \cA_{\bF}(\mathcal X)$, let $X^\xi$ be the corresponding state process. Then it holds that,
\begin{equation} \label{T0}
	\lim_{s\nearrow T} \bE \left[X^\xi_sY_s |\mathcal X\right] = 0.
\end{equation}
Indeed,  since $A \in \cM_{-1}$, for any $0 \leq s < T$
\begin{equation*}
\begin{split}
	\left| \bE\left[X^\xi_sY_s |\mathcal X\right] \right|&= \left| \bE\left[X^\xi_s (X_sA_s+B_s) |\mathcal X\right]  \right|\\
	&\leq \frac{C}{T-s}\bE\left[(X^\xi_s)^2+(X_s)^2 |\mathcal X\right]+\bE\left[|X^\xi_sB_s| |\mathcal X\right]\\
	&=\frac{C}{T-s}\bE\left[\left(\int_s^T\xi_u\,du\right)^2+\left(\int_s^T\xi_u^*\,du\right)^2\bigg|\mathcal X\right]+\bE\left[|X^\xi_sB_s| |\mathcal X\right]\\
	&\leq C\bE\left[\int_s^T\xi_u^2\,du+\int_s^T(\xi_u^*)^2\,du\bigg|\mathcal X\right]+\bE\left[|X^\xi_sB_s||\mathcal X\right]\xrightarrow{s\nearrow T}0.
\end{split}
\end{equation*}

With this, we can now show that $\xi^*$ is a best response against $\mu^*$. In fact, for each $\epsilon > 0$,
the convexity of the Hamiltonian yields
\begin{equation*}
\begin{split}
   &~\bE\left[\int_0^{T-\epsilon}\left(\kappa_s\mu^*_sX^\xi_s+\eta_s\xi_s^2+\lambda_s(X^\xi_s)^2\right)\,ds\bigg|\mathcal X\right]-\bE\left[\int_0^{T-\epsilon}\left(\kappa_s\mu^*_sX_s+\eta_s(\xi_s^*)^2 +\lambda_s (X_s)^2\right)\,ds\bigg|\mathcal X\right]\\
     =&~ \bE\left[\int_0^{T-\epsilon} \left(H(s,\xi_s,X^\xi_s,Y_s;\mu^*)-H(s,\xi^*_s,X_s,Y_s;\mu^*)+(\xi_s-\xi^*_s)Y_s\right)\,ds\bigg|\mathcal X\right]\\
      \geq&~\mathbb E\left[\left.\int_0^{T-\epsilon}(H(s,\xi^{*}_s,X^{\xi}_s,Y_s;\mu^*)-H(s,\xi^*,X_s,Y_s;\mu^*)+(\xi_s-\xi^*_s)Y_s)\,ds\right|\mathcal X\right]\\
   \geq &~ \bE\left[\int_0^{T-\epsilon} \left((\kappa_s\mu^*_s+2\lambda_sX_s)(X^\xi_s-X_s)+(\xi_s-\xi^*_s)Y_s\right)\,ds \bigg|\mathcal X\right].
\end{split}
\end{equation*}
Furthermore, integration by part implies that for any $\epsilon > 0$,
\begin{equation}\label{eq:tech_for_value_fct}
    \begin{split}
& Y_{T-\epsilon}(X_{T-\epsilon}-X^\xi_{T-\epsilon}) \\ =&~Y_0(X_0-X^\xi_0)+\int_0^{T-\epsilon} (X_s-X^\xi_s)\,dY_s+\int_0^{T-\epsilon} Y_s\,d(X_s-X^\xi_s)\\
    =&~-\int_0^{T-\epsilon}(\kappa_s\mu^*_s+2\lambda_sX_s)(X_s-X^\xi_s)\,ds+\int_0^{T-\epsilon}Z^Y_s(X_s-X^\xi_s)\,d\widetilde W_s\\
    &~-\int_0^{T-\epsilon} Y_s(\xi^*_s-\xi_s)\,ds.
    \end{split}
\end{equation}
Therefore,
\begin{eqnarray*}
    &&\bE\left[\int_0^{T-\epsilon}\left(\kappa_s\mu^*_sX^\xi_s+\eta_s\xi_s^2+\lambda_s(X^\xi_s)^2\right)\,ds\bigg|\mathcal X\right]-\bE\left[\int_0^{T-\epsilon}\left(\kappa_s\mu^*_sX_s+\eta_s(\xi_s^*)^2+\lambda_s(X_s)^2\right)\,ds\bigg|\mathcal X\right]\\
     &&\quad \geq \bE \left[  Y_{T-\epsilon}(X_{T-\epsilon}-X^\xi_{T-\epsilon}) \bigg|\mathcal X\right].
\end{eqnarray*}
The equation \eqref{T0} does indeed yield
$$\lim_{\epsilon \to 0}  \bE \left[  Y_{T-\epsilon}(X_{T-\epsilon}-X^\xi_{T-\epsilon}) |\mathcal X \right] = 0.$$
Using the Lebesgue convergence theorem and taking $\eps\rightarrow0$, we obtain
\[
\bE \left[ \int_0^T\left( \kappa_sX_s\mu^*_s+\eta_s\xi^2_s+\lambda_s X_s^2\right)\,ds \bigg|\mathcal X\right]-\bE \left[ \int_0^T\left( \kappa_sX_s\mu^*_s+\eta^*_s\xi^2_s+\lambda_s (X_s)^2\right)\,ds \bigg|\mathcal X\right]\geq 0.
\]
In other words $ J(\mathcal X,\xi;\mu^*)-J(\mathcal X,\xi^*;\mu^*)\geq 0$. Finally, \eqref{T0} and \eqref{eq:tech_for_value_fct} again yield
$$V(\mathcal X;\mu^*)=\frac{1}{2}A_0 {\mathcal X}^2+\frac{1}{2}B_0\mathcal X + \frac{1}{2} \bE \left[ \int_0^{T} \kappa_sX_s\mu^*_s\,ds \bigg| \mathcal X\right].$$
Now assume $\mu'$ is another equilibrium, i.e. there is an optimal control $\xi'$ such that $J(\mathcal X,\xi;\mu')\geq J(\mathcal X,\xi';\mu')$ for any $\xi$, and $\mu'=\mathbb E[\xi'|\mathcal F^0]$. Note that $J(\mathcal X,\xi;\mu')$ is strictly convex for $\xi$. Thus, there is a unique optimal control, which must satisfy $\xi'=Y'/2\eta$, where $(X',Y')$ is the solution to \eqref{optimality-system-PVE} with $\mu$ replaced by $\mu'$. By the uniqueness of the solution of \eqref{conditional-MF-FBSDE}, it must hold that $\mu'=\mu^*$ as well as $\xi'=\xi^*$.
\end{proof}

\begin{remark}
If we suppose that $\mathcal X=x$ is a deterministic initial value of the state process at time $\tau > 0$, then we can define the space of admissible controls as
\[
	\mathcal A_{\mathbb Q}(\tau,x):= \left\{  \xi \in L^2_{\mathbb Q}(\tau,T): ~\int_\tau^T \xi_s\,ds = x\right\}
\]	
where $\mathbb Q=(\mathcal Q_t)_{0\leq t\leq T}$ is the filtration generated by $W$ and $W^0$. Assuming that the cost coefficients satisfy Assumption \ref{ass-PVE} with $\bF$ replaced by $\mathbb Q$, the same arguments as before show that the FBSDE
\begin{equation*}
\left\{\begin{aligned}
X_s =&x -\int_\tau^s \dfrac{Y_t}{2\eta_t}\,dt,\quad X_T=0,\\
Y_s =&Y_r + \int_s^r \left(\kappa_t\bE\left[\left.\frac{Y_t}{2\eta_t}\right|\cF^0_t\right]+2\lambda_tX_t\right)\,dt+\int_s^r Z_t\,d\widetilde{W}_t,~r<T
\end{aligned}
\right.
\end{equation*}
has a unique solution $(X,Y=AX+B,Z)$ with $(X,B)\in\mathcal H_\alpha\times\mathcal H_\gamma$ and $\mu^* = \bE (Y/(2\eta)|\cF^0)$ is {a} solution of the MFG starting at time $\tau$. Moreover the value function is given by:
    $$V(\tau,x;\mu^*)=\frac{1}{2}A_\tau x^2+\frac{1}{2}B_\tau x + \frac{1}{2} \bE \left[ \int_\tau^{T} \kappa_sX_s\mu^*_s\,ds \bigg| \mathcal Q_\tau\right].$$
Since $(X,B)\in\mathcal H_\alpha\times\cH_\gamma$ and $\xi^*\in\mathcal A_{\mathbb Q}(\tau,x)$,
\begin{equation*}
    \begin{split}
    &~B_\tau x+\bE\left[\int_\tau ^T\kappa_s\mu^*_sX_s\,ds\bigg|\mathcal{Q}_\tau \right]\\
    \leq&~ x(T-\tau )^\gamma\sup_{\tau \leq t\leq T}\left|\frac{B_t}{(T-t)^\gamma}\right|+ \kappamax(T-\tau )^\alpha\bE\left[\left.\int_\tau ^T|\mu^*_s|\,ds\sup_{\tau \leq t\leq T}\left|\frac{X_t}{(T-t)^\alpha}\right|\right|\mathcal Q_\tau \right]\xrightarrow{\tau \nearrow T} 0.
    \end{split}
\end{equation*}
Since $A_\tau  \to +\infty$ as $\tau $ tends to $T$, we get the following terminal condition for the value function:
\begin{equation*}
\lim_{\tau \uparrow T}V(\tau ,x;\mu)=
\left\{\begin{array}{ll}
0, \quad x=0;\\
\infty,\quad x \neq  0.
\end{array}\right.
\end{equation*}
\end{remark}

\subsection{Common information environments}\label{sec-CVE}

The benchmark case where all players share the same information, except for their initial value can be analyzed in greater detail. In this section we therefore assume that all  randomness is generated by the common Brownian motion $W^0$ and the initial value $\mathcal X$.

\begin{ass}\label{ass-CVE}
The processes $\kappa$, $\lambda$, $\eta$ and $1/\eta$ belong to $L^{\infty}_{\bF^0}([0,T]\times\Omega;[0,\infty))$.
\end{ass}
The weak interaction condition \eqref{ineq-PVE} is not required in this section. Under the common information assumption the conditional mean-field FBSDE \eqref{conditional-MF-FBSDE} reduces to the following FBSDE:
\begin{equation}\label{FBSDE-CVE}
\left\{\begin{aligned}
dX_t &=-\frac{Y_t}{2\eta_t}\,dt,\\
-dY_t &=\left(\frac{\kappa_t}{2\eta_t}\bE \left[ Y_t \big| \cF^0_t\right]+2\lambda_tX_t\right)\,dt-Z_t\,dW^0_t,\\
X_0 &=\mathcal X, \\
 X_T & =0.
\end{aligned}
\right.
\end{equation}

\subsubsection{Common initial portfolio}\label{sec:common-initial}
%----------------------

In this subsection we further assume that the initial portfolio is common to all players, i.e. $\mathcal X=x\in\mathbb R$. In this case all processes are $\bF^0$-adapted
and the mean-field FBSDE \eqref{FBSDE-CVE} simplifies to the regular FBSDE
\begin{equation}\label{FBSDE-CVE-CIV}
\left\{\begin{aligned}
dX_t &=-\frac{Y_t}{2\eta_t}\,dt,\\
-dY_t &=\left(\frac{\kappa_tY_t}{2\eta_t}+2\lambda_tX_t\right)\,dt-Z_t\,dW^0_t,\\
X_0 &=x, \\
 X_T & =0.
\end{aligned}
\right.
\end{equation}
In this setting, we can check that $Y$ is given by $Y=A^\kappa X$ where
\begin{equation}\label{Y-AX}
-dA^\kappa_t=\left(2\lambda_t+\frac{\kappa_tA^\kappa_t}{2\eta_t}-\frac{(A^\kappa_t)^2}{2\eta_t}\right)\,dt-Z^{A^\kappa}_tdW^0_t, ~~A^\kappa_T=\infty.
\end{equation}
This singular terminal condition on $A^\kappa$ is necessary to satisfy the constraint $X_T=0$. This equation has a unique solution, due to Corollary \ref{cor-appendix} in the appendix.
By \eqref{FBSDE-CVE-CIV},
\[
    X_t=xe^{-\int_0^t\frac{A^\kappa_r}{2\eta_r}\,dr}.
\]
The candidate of the optimal strategy is $\xi^*=Y/2\eta$, where $Y$ is the solution to \eqref{FBSDE-CVE-CIV}. Since both $Y$ and $\eta$ are $\mathbb F^0$-adapted, the consistency condition \eqref{eq:consistency_condition} reads 
$ \mu^*=\mathbb E[\xi^*|\mathcal F^0]=\xi^*$. 

\begin{lemma}\label{admissible-control}
Under Assumption \ref{ass-CVE}, the processes $A^\kappa$, $X$, $Y=A^\kappa X$ and $\xi^*=\mu^*=\frac{Y}{2\eta}$ have the same sign as $x$. Moreover
$$A^\kappa\in \cM_{-1},\ X \in \cM_\alpha, \ Y\in \cM_{\alpha-1},\ \xi^* \in \cM_{\alpha-1}.$$
\end{lemma}
\begin{proof}
Let $\widetilde{A^\kappa_t}=A^\kappa_t e^{\int_0^t\frac{\kappa_s}{2\eta_s}\,ds}$. 
Due to Lemma \ref{existence-asymptotic-expansion-A} in the appendix, the following estimate holds for any $0 \leq t < T$:
\[
    \frac{1}{\bE\left[\left.\int_t^T\frac{1}{2\eta_s}e^{-\int_0^s\frac{\kappa_r}{2\eta_r}\,dr}\,ds\right|\mathcal{F}^0_t\right]}\leq {\widetilde{A^\kappa_t}}.
\]
Hence the process $A^\kappa_t$ is bounded from below by:
\begin{equation}\label{kappa-infty-A}
    \begin{split}
        A^\kappa_t \geq &~ \frac{e^{-\int_0^t\frac{\kappa_r}{2\eta_r}\,dr}}{\bE\left[\left.\int_t^T\frac{1}{2\eta_s}e^{-\int_0^s\frac{\kappa_r}{2\eta_r}\,dr}\,ds\right|\mathcal{F}^0_t\right]} = \frac{1}{\bE\left[\left.\int_t^T\frac{1}{2\eta_s}e^{-\int_t^s\frac{\kappa_r}{2\eta_r}\,dr}\,ds\right|\mathcal{F}^0_t\right]}
        \geq 
        2\etamin \frac{1}{(T-t)}.
    \end{split}
\end{equation}
Hence \eqref{expA} holds:
\begin{equation*}
e^{-\int_0^t\frac{A^\kappa_r}{2\eta_r}\,dr} \leq \exp\left(-2\etamin \int_0^t\frac{1}{2\eta_r(T-r)}\,dr \right)  \leq \left( \frac{T-t}{T} \right)^{\alpha}.
\end{equation*}
The conclusion on $X$ can be deduced immediately. Again from Lemma \ref{existence-asymptotic-expansion-A} in the appendix, ${\widetilde{A^\kappa}}$ is bounded from above:
\[
{\widetilde {A^\kappa_t}} \leq\frac{1}{(T-t)^2}\bE\left[\left.\int_t^T\left(2\eta_se^{\int_0^s\frac{\kappa_r}{2\eta_r}\,dr}+2(T-s)^2\lambda_se^{\int_0^s\frac{\kappa_r}{2\eta_r}\,dr}\right)\,ds\right|\mathcal{F}^0_t\right].
\]
Thus we get an upper bound on $A^\kappa$:
\begin{equation*}
    \begin{split}
        A^\kappa_t \leq & ~ \frac{e^{-\int_0^t\frac{\kappa_r}{2\eta_r}\,dr}}{(T-t)^2}\bE\left[\left.\int_t^T\left(2\eta_se^{\int_0^s\frac{\kappa_r}{2\eta_r}\,dr}+2(T-s)^2\lambda_se^{\int_0^s\frac{\kappa_r}{2\eta_r}\,dr}\right)\,ds\right|\mathcal{F}^0_t\right] \\
        \leq &~\frac{2}{(T-t)^2}\mathbb E\left[\left.\etamax e^{\int_0^T\frac{\kappa_r}{2\eta_r}\,dr}(T-t)+\frac{1}{3}\lambdamax e^{\int_0^T\frac{\kappa_r}{2\eta_r}\,dr}(T-t)^3\right|\mathcal F^0_t\right] \\
        \leq &~\frac{2}{(T-t)} e^{\frac{\kappamax T}{2\etamin}}\left[\etamax+\frac{\lambdamax T^2}{3}\right].
    \end{split}
\end{equation*}
Collecting all inequalities we get that $A^\kappa \in \cM_{-1}$ and
\begin{equation*}
    \begin{split}
        |\xi^*_t |=&~ \frac{A^\kappa_t |X_t|}{2\eta_t}  = |x| \frac{A^\kappa_te^{-\int_0^t\frac{A^\kappa_r}{2\eta_r}\,dr}}{2\eta_t}\\
        \leq&~\frac{|x|}{\etamin T^\alpha} \left[\etamax+\frac{\lambdamax T^2}{3}\right] e^{\frac{\kappamax T}{2\etamin}}\left( T-t \right)^{\alpha-1}.
    \end{split}
\end{equation*}
A similar inequality holds for $Y$. 
\end{proof}

It follows from the preceding lemma that $Y$ is a non-negative or non-positive supermartingale so the limit of $Y$ at the terminal time $T$ exists and is finite. Since $X \in \cM_\alpha$, we deduce that $\lim_{t\nearrow T} Y_t X_t = 0$. Moreover, the process $Z$ belongs to $L_{\bF^0}^p([0,T-]\times \Omega;\mathbb R)$  for any $p$.

The following theorem verifies that $\xi^*$ is optimal. The proof is the similar to Proposition \ref{verification-PVE}.
\begin{theorem}\label{verification-CVE}
Under Assumption \ref{ass-CVE} and if the initial value is deterministic, $\xi^*(=\mu^*)$ is {the unique} optimal control as well as the equilibrium to MFG \eqref{MFG-PVE}. Moreover the value function is given by:
\begin{equation} \label{eq:expression_value_fct}
   V(x;\mu^*)=\frac{1}{2}A^\kappa_0 x^2+\frac{1}{2}\bE\left[\int_0^T\kappa_s\mu^*_sX_s\,ds\right],
\end{equation}
and is non-negative.
\end{theorem}

\subsubsection{Private initial portfolio}\label{sec:private-initial}
%----------------------

Let us now return to the problem \eqref{FBSDE-CVE}. Theorem \ref{main-result} implies there exists a unique soluton $(X,Y,Z)$ to \eqref{FBSDE-CVE}. From the solution to \eqref{FBSDE-CVE-CIV}, we deduce that
\[
	\mu^*_t = \frac{1}{2\eta_t}\bE [Y_t\big| \cF^0_t]= \frac{\bE[\mathcal X]}{2\eta_t}A^\kappa_te^{-\int_0^t\frac{A^\kappa_r}{2\eta_r}\,dr},
\]
where $A^\kappa$ solves the BSDE \eqref{Y-AX}. Since $Y$ is given by $Y=AX+B$, %\sout{where $A=A^0$,} 
we obtain (see equation \eqref{Riccati-system-PVE} {in Section \ref{existence-verification}}) that
\begin{equation}\label{Riccati-system-PVE-bis}
\left\{\begin{aligned}
    -dA_t =&\left(2\lambda_t-\frac{A^2_t}{2\eta_t}\right)\,dt-Z^A_t\,dW^0_t, \quad A_T = +\infty\\
    -dB_t =&\left(\kappa_t \mu^*_t-\frac{A_tB_t}{2\eta_t}\right)\,dt-Z^B_t\,dW^0_t, \quad B_T= 0.
\end{aligned}\right.
\end{equation}
Note that $A$ and $B$ are $\bF^0$-adapted. % \sout{and that $A=A^0$}. 
Thereby we have an explicit solution: for $t \in [0,T]$
\begin{equation*}
\left\{\begin{aligned}
    B_t=&~\bE\left[\left.\int_t^T \kappa_s \mu^*_s e^{-\int_t^s (2\eta_r)^{-1}A_r\,dr}\,ds\right|\mathcal F^0_t\right], \\
	X_t=&~\mathcal Xe^{-\int_0^t (2\eta_r)^{-1}A_r\,dr}-\int_0^t(2\eta_s)^{-1}B_se^{-\int_s^t (2\eta_r)^{-1} A_r\,dr}\,ds, \\
	Y_t = & A_t X_t + B_t.
\end{aligned}\right.
\end{equation*}
Again from the general analysis of Section \ref{existence-verification}, the system \eqref{Riccati-system-PVE-bis} has a unique solution; similar arguments as in the proof of Proposition \ref{verification-PVE} can be applied to verify that the optimal state process for a given initial position $\mathcal X=x\in\mathbb R$ is given by:
\begin{equation}\label{position-private-initial}
X^{*,x}_t = (x-\bE[\mathcal X])e^{-\int_0^t (2\eta_r)^{-1}A_r\,dr} + \bE[\mathcal X] e^{-\int_0^t (2\eta_r)^{-1}A^\kappa_r\,dr} .
\end{equation}
Thus, if different players hold different initial portfolios, then a trader's optimal position consists of a weighted sum of the competitors' average portfolio size $\mathbb E[\mathcal X]$ and the deviation of the own initial position from that average.

\begin{remark}
By \cite[Theorem 2.4]{Pardoux1999} the unique solution $A^{\kappa,n}$ to the BSDE
\[
    -dA^{\kappa,n}_t=\left(2\lambda_t+\frac{\kappa_tA^{\kappa,n}_t}{2\eta_t}-\frac{(A^{\kappa,n}_t)^2}{2\eta_t}\right)\,dt-Z^{A^{\kappa,n}}_tdW^0_t, ~~A^{\kappa,n}_T=2n
\]
is increasing in $\kappa$. By Lemma \ref{lem:estim_H_n_1} in the appendix, this result carries over to the process $A^{\kappa}$. In particular, $A^\kappa \geq A$. Moreover $A^\kappa_0 > A_0$ if $\kappa > 0$ on some set of positive measure.
\end{remark}

The preceding remark shows that the dependence of the optimal portfolio process {on $\kappa$} decreases if $\mathbb E[\mathcal X]>0$. It also suggests that
- contrary to the previous case - the sign of the optimal portfolio process $X^*$ may change on the interval $[0,T]$. In fact, if $\bE[\mathcal X] > 0$ and $x \geq \bE[\mathcal X]$, then $X^{*,x}$ remains non-negative on $[0,T]$. However, if $0 < x < \zeta \bE[\mathcal X]$ where $\zeta :=1-\exp\left(\frac{A_0-A^\kappa_0}{ 2\etamax} t \right) >0$, then $X^{*,x}$ becomes negative shortly after the initial time; see also Figure \ref{fig:graph_state_bis} below.

%%%%%%%%%%%%%%%%%%%%%%%%%%%%%%%%%%%%%%%%%%%%%%%%%%%%%%%%%%%%%%%%%%%%%%%%%%%%%
%%%%%%%%%%%%%%%%%%%%%%%%%%%%%%%%%%%%%%%%%%%%%%%%%%%%%%%%%%%%%%%%%%%%%%%%%%%%%
%%%%%%%%%%%%%%%%%%%%%%%%%%%%%%%%%%%%%%%%%%%%%%%%%%%%%%%%%%%%%%%%%%%%%%%%%%%%%

\subsubsection{Constant cost coefficients}

In this section, we consider a deterministic benchmark example that can be solved explicitly. 
\begin{ass}
The processes $\lambda$, $\kappa$, $\eta$ are positive constants.
\end{ass}

Under the preceding assumption, the Riccati equation \eqref{Y-AX} reduces to
\begin{equation*}
    -dA^\kappa_t=\left(2\lambda+\frac{\kappa A^\kappa_t}{2\eta}-\frac{(A^\kappa_t)^2}{2\eta}\right)\,dt, ~~A^\kappa_T=\infty.
\end{equation*}
Its explicit solution is given by
\[
	A^\kappa_t=2\eta\gamma\coth\left(\gamma(T-t)\right)+\frac{\kappa}{2}
\]
where
\[
	\gamma:=\sqrt{\frac{\lambda}{\eta}+\frac{\kappa^2}{16\eta^2}}.
\]
If all players share the same initial portfolio (see Subsection \ref{sec:common-initial}), then {the optimal portfolio process is given by}
\begin{equation} \label{eq-det-MFG-optimal-X}
	X_t^*=\exp\left(-\frac{\kappa}{4\eta}t\right)\frac{\sinh(\gamma(T-t))}{\sinh(\gamma T)} x
\end{equation}
and the optimal liquidation rate is given by
\begin{align*}
	\xi^*_t&= \left(\gamma \coth(\gamma(T-t))+\frac{\kappa}{4\eta}\right)X_t^*\\
	&=\exp\left(-\frac{\kappa}{4\eta}t\right)\left(\frac{ \gamma\cosh(\gamma(T-t))}{\sinh(\gamma T)}+\frac{\kappa\sinh(\gamma(T-t))}{4\eta\sinh(\gamma T)}\right)x.
\end{align*}

When $\kappa\rightarrow 0$, then $\xi^*_t\rightarrow \frac{{\widetilde\gamma}\cosh({\widetilde\gamma}(T-t))}{\sinh({\widetilde\gamma} T)}x$ with ${\widetilde\gamma}=\sqrt{\frac{\lambda}{\eta}}$. This corresponds to the benchmark model in \cite{AC-2001}. This convergence can also be seen from Figure 1. Furthermore, we see that---as in the corresponding single player models---the optimal liquidation rate is always positive, i.e., round trips are not beneficial. Moreover, we notice that the portfolio process \eqref{eq-det-MFG-optimal-X} corresponds to the optimal portfolio process in an Almgren--Chriss model with adjusted risk aversion $\widetilde \lambda =\lambda +\frac{\kappa^2}{16\eta}$ and with additional exponential decay of rate $\frac{\kappa}{4\eta}$.

When $\kappa \to \infty$, then $\xi^*_0\rightarrow \infty$ while $\xi^*_t\rightarrow 0$ for $t>0$. That is, when the impact of interaction is very strong,
then the players trade very fast initially and very slowly afterwards. The intuitive reason is that in this case an individual player would benefit from trading fast slightly before his competitors start trading in order to avoid the negative drift generated by the mean-field interaction. As all the players are statistically identical, they ``coordinate'' on an equilibrium trading strategy as depicted in Figure 2. Thus, our model provides a possible explanation for large price increases or decreases in markets with strategically interacting players with similar preferences.

\begin{figure}
\centering
\includegraphics[width=7.0cm]{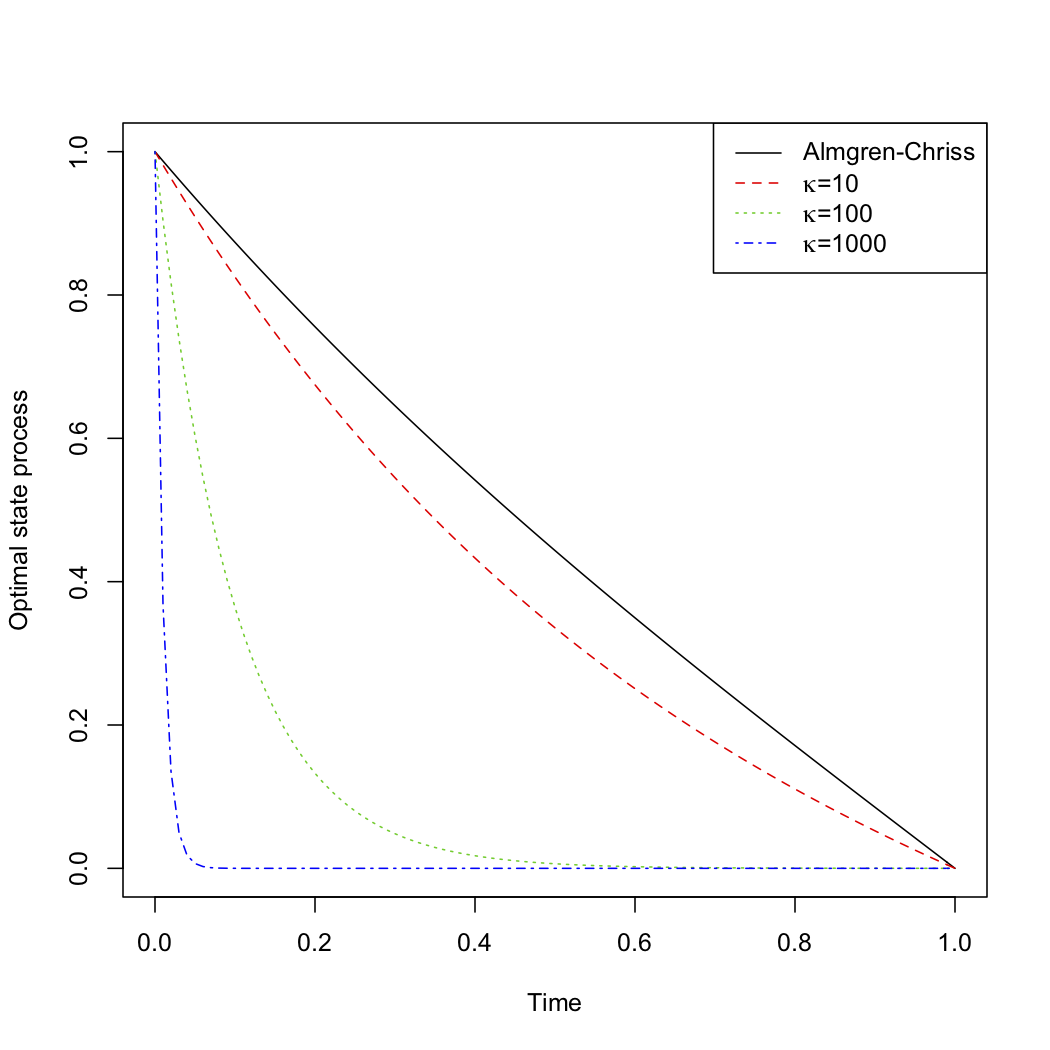}
\includegraphics[width=7.0cm]{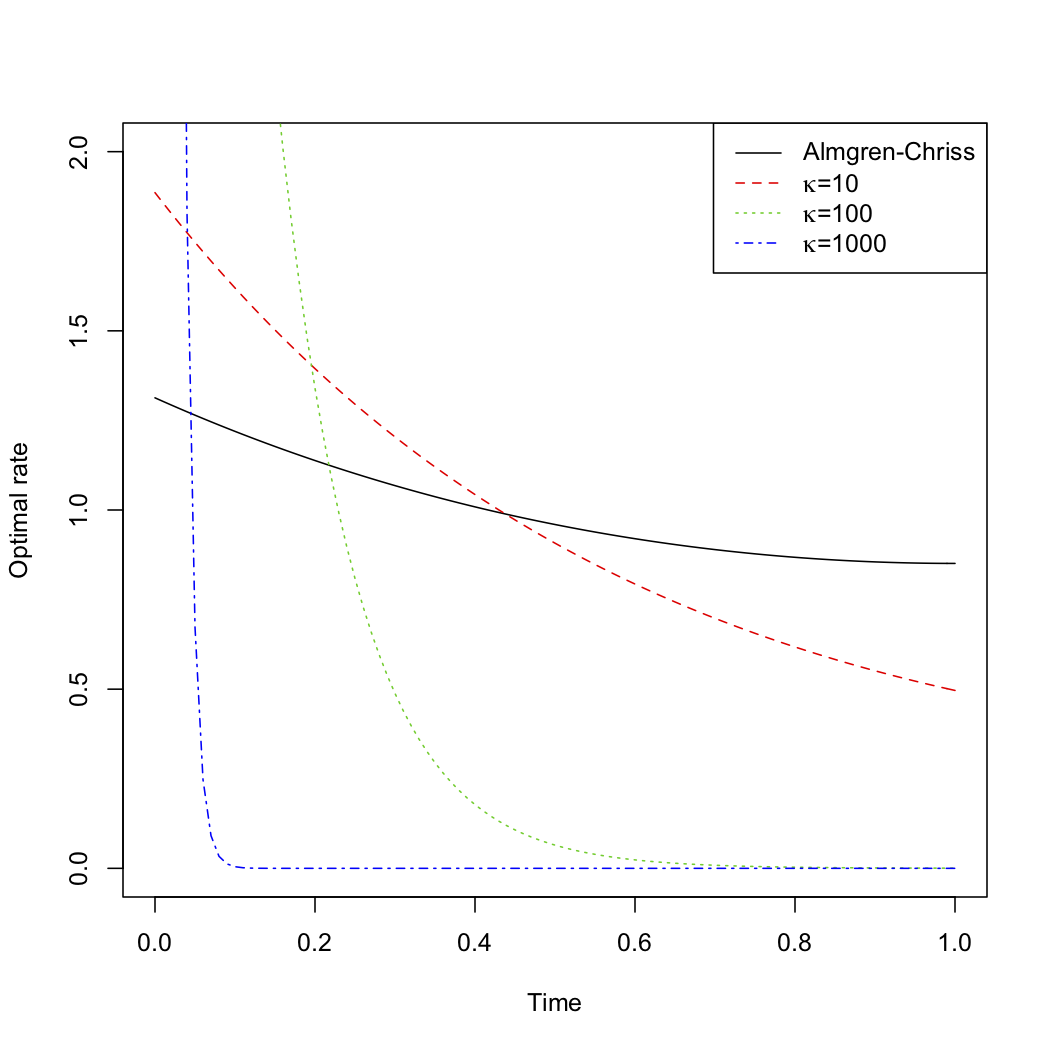}
\caption{Current state $X^*$ (left) and optimal liquidation rate $\xi^*$ (right) corresponding to parameters $T=1$, $x=1$, $\lambda=5$ and $\eta=5$. The solid line corresponds to $\kappa=0$, that is the Almgren-Chriss model with temporary impact.}
\label{fig:graph_rate}
\end{figure}

If the players hold different initial portfolios (Subsection \ref{sec:private-initial}), then \eqref{position-private-initial} shows that the optimal portfolio process is given by
\begin{equation*} 
X_t^{*,x}= (x-\bE[\mathcal X])\frac{\sinh(\widetilde \gamma(T-t))}{\sinh(\widetilde \gamma T)}+ \exp\left(-\frac{\kappa}{4\eta}t\right)\frac{\sinh(\gamma(T-t))}{\sinh(\gamma T)} \bE[\mathcal X].
\end{equation*}
Figure \ref{fig:graph_state_bis} confirms that the sign of $X^{*,x}$ is indeed changing when $x$ is small.

\begin{figure}
\centering
\includegraphics[width=7.5cm]{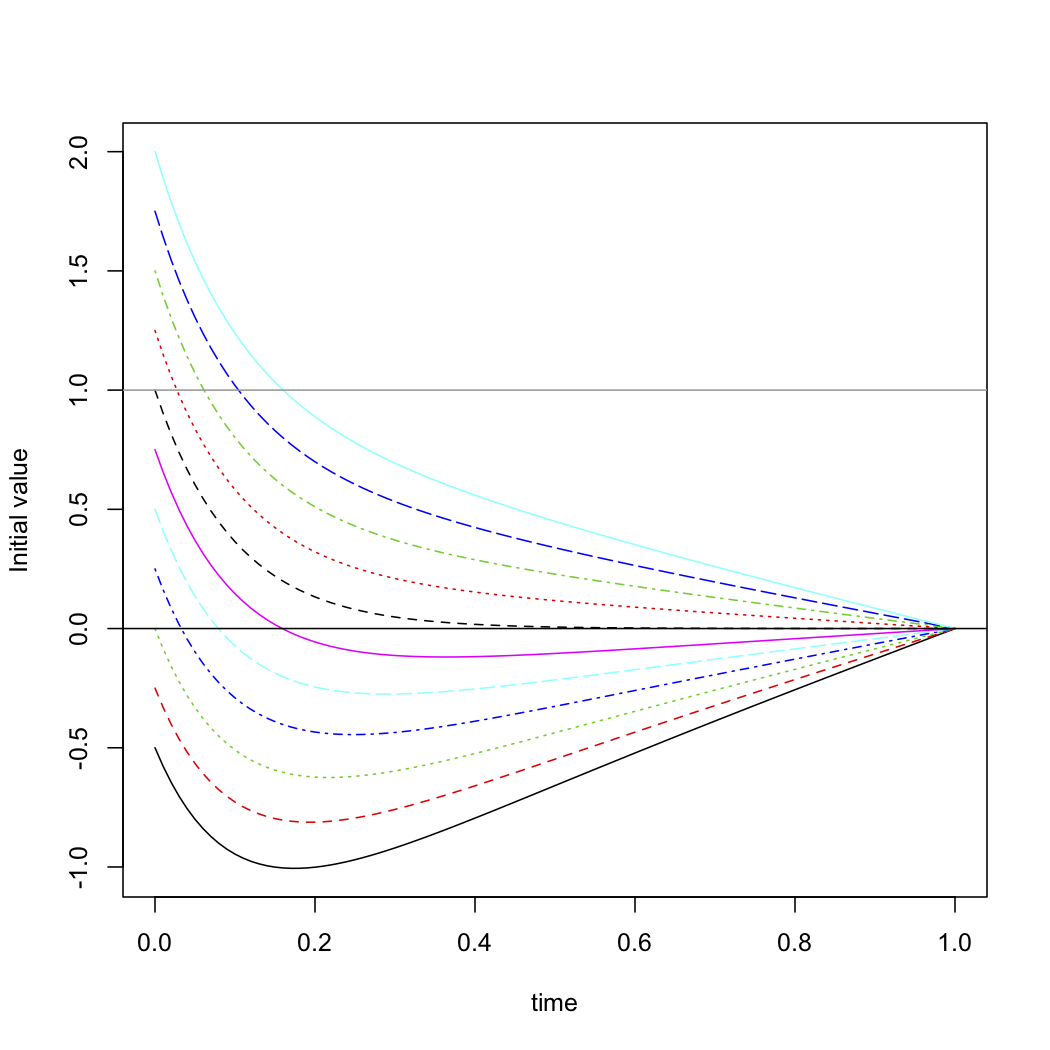}
\caption{Current state $X^{*,x}$ corresponding to parameters $T=1$, $\bE[\mathcal X]=1$, $\lambda=5$, $\eta=5$ and $\kappa=100$ for different values of the initial portfolio $x$.}
\label{fig:graph_state_bis}
\end{figure}

%%%%%%%%%%%%%%%%%%%%%%%%%%%%%%%%%
%%%%%%%%%%%%%%%%%%%%%%%%%%%%%%%%%
%%%%%%%%%%%%%%%%%%%%%%%%%%%%%%%%%
%%%%%%%%%%%%%%%%%%%%%%%%%%%%%%%%%

%%%%%%%%%%%%%%%%%%%%%%%%%%%%%%%%%%%
%%%%%%%%%%%%%%%%%%%%%%%%%%%%%%%%%%%
%%%%%%%%%%%%%%%%%%%%%%%%%%%%%%%%%%%

%%%%%%%%%%%%%%%%%%%%%%%%%%%%%
%%%%%%%%%%%%%%%%%%%%%%%%%%%%%
%%%%%%%%%%%%%%%%%%%%%%%%%%%%%

\section{Approximate Nash Equilibrium}
In this section we show that an $\epsilon$-Nash equilibrium for the $N$ player portfolio liquidation game can be constructed from the solution to the MFG \eqref{MFG-PVE} when the number of players is large if all players share the same cost structure.

\begin{ass}\label{ass-Epsilon-i}
Assume for any $i=1,\cdots, N$, $\kappa^i$, $\eta^i$ and $\lambda^i$ admit the following expression
\[
	\kappa^i_t=\kappa(t,\mathcal X^i,W^i_{\cdot\wedge t},W^0_{\cdot\wedge t}),\quad\eta^i_t=\eta(t,\mathcal X^i,W^i_{\cdot\wedge t},W^0_{\cdot\wedge t}),\quad\lambda^i_t=\lambda(t,\mathcal X^i,W^i_{\cdot\wedge t},W^0_{\cdot\wedge t})
\]
for some non-negative deterministic bounded and measurable functions $\kappa$, $\eta$ and $\lambda$.
\end{ass}

The next result is an adaptation to the Yamada-Watanabe result for FBSDE. The proof follows from the same arguments given in, e.g. \cite{carm:dela:18b} and \cite{Bahlali2007}.

\begin{lemma} \label{lem:yamada_watanabe}
There exists a measurable function $\Phi : \bR \times (\mathcal C[0,T])^2 \to  \mathcal H_\alpha  \times (\mathcal C[0,T-])^2$ such that for any $i=1,\cdots,N$
$$\left(X^i_t,Y^i_t,\int_0^t Z^{i} ds\right)_{0\leq t \leq T} = \Phi(\mathcal X^i,W^i,W^0),$$
where $\left(X^i,Y^i, Z^{i} \right)$ is the solution to FBSDE \eqref{conditional-MF-FBSDE} associated with $(W^0,\mathcal X^i,W^i,\kappa^i,\eta^i,\lambda^i)$.
In particular, there exists a function $\phi$ independent of $(\mathcal X^1,\cdots,\mathcal X^N,W^0,W^1,\ldots,W^N)$ such that
\begin{equation}\label{eq:approximate-nash}
    \xi^{*,i}=\phi(\mathcal X^i,W^0,W^i),
\end{equation}
where $\xi^{*,i}$ is an optimal control for agent $i$ associated with $(W^0,\mathcal X^i,W^i,\kappa^i,\eta^i,\lambda^i)$, given by \eqref{optimal-candidate-control-PVE}.

\end{lemma}

In view of the {above} lemma, under Assumption \ref{ass-Epsilon-i} each player's unique best response $\xi^{*,i}$ to the mean-field equilibrium $\mu^*$ can be represented in terms of the function $\phi$ as {in \eqref{eq:approximate-nash}}.
In particular, each individual action has the same distribution as the mean-field equilibrium:
\begin{equation}\label{action=equilibrium}
    \mu^*_t=\mathbb E[\xi^{*,i}_t|\mathcal F^0_t],\quad \textrm{a.s. a.e.}
\end{equation}
Proposition \ref{existence_MF_FBSDE} guarantees the existence of a constant $C$ such that
\begin{equation} \label{eq:L2_bound}
\bE\left[\int_0^T|\xi^{*,i}_t|^2\,dt \right] \leq C,
\end{equation}
and Lemma \ref{lem:yamada_watanabe} yields a real-valued function $\psi$, which is independent of $i$, such that
\begin{equation} \label{eq:expected_L2_bound}
\bE\left[\int_0^T|\xi^{*,i}_t|^2\,dt\bigg| \mathcal X^i=x^i \right]=\psi(x^i),
\end{equation}

Before we prove the main result of this section, we recall the cost functional $J^{N,i}\left(\vec{\xi}\right)$ from \eqref{cost-player-i}.

\begin{theorem}\label{epsilon-Nash-equilibrium}
Assume that Assumption \ref{ass-Epsilon-i} is satisfied and that the admissible control space for each player $i=1, \ldots, N$ is given by
$$\mathcal{A}^i:=\left\{\xi\in\mathcal A_{\mathbb F^i}(x^i): \bE\left[\int_0^T|\xi_t|^2\,dt\bigg| \mathcal X^i=x^i \right]\leq M(x^i) \right\}$$
for some fixed positive function $M$ such that $\psi \leq M$.
Then, for each $1\leq i\leq N$ and each $\xi^i\in\mathcal{A}^i$,
\[
    J^{N,i}\left(\vec{\xi^{*}}\right)\leq J^{N,i}(\xi^i, \xi^{*,-i})+O\left(\frac{1}{\sqrt N}\right),
\]
where {$\xi^{*,i}$ is given by \eqref{eq:approximate-nash}}, $(\xi^i, \xi^{*,-i})=(\xi^{*,1},\cdots,\xi^{*,i-1},\xi^{i},\xi^{*,i+1},\cdots,\xi^{*,N})$ and $O\left(\frac{1}{\sqrt N}\right)$ is to be interpreted as $\frac{g(x_i)}{\sqrt N}$ for some real-valued function $g$ independent of $i$.
\end{theorem}
\begin{proof}
By the symmetry of the $N$ player game, it is sufficient to show the result for Player $1$. We first estimate the following term:
\begin{equation*}
    \begin{split}
        &~\bE\left[\int_0^T\left(\mu^{*}_t-\frac{1}{N}\sum_{j=1}^N\xi^{*,j}_t\right)^2\,dt \Bigg| \mathcal X^1=x^1 \right] \\
         =&~\frac{1}{N^2}  \bE\left[\int_0^T \sum_{i\neq j} \left( \mu^{*}_t-\xi^{*,i}_t \right)\left( \mu^{*}_t-\xi^{*,j}_t \right) \,dt\Bigg| \mathcal X^1=x^1\right] +  \frac{1}{N^2}  \bE\left[\int_0^T \sum_{ j=1}^N \left( \mu^{*}_t-\xi^{*,j}_t \right)^2 \,dt\Bigg| \mathcal X^1=x^1\right] . 
    \end{split}
\end{equation*}
Using \eqref{eq:L2_bound} and \eqref{eq:expected_L2_bound}, the second term is bounded by $\dfrac{2(M(x^1)+(2N-1)C)}{N^2}$. For the first term, if $i,j \neq 1$, then the conditional expectation reduces to the expectation and since $\xi^{*,i}$ and $\xi^{*,j}$ are conditionally independent given $W^0$ for $i \neq j$,
\begin{eqnarray*}
\bE\left[\int_0^T  \left( \mu^{*}_t-\xi^{*,i}_t \right)\left( \mu^{*}_t-\xi^{*,j}_t \right) \,dt\Bigg| \mathcal X^1=x^1\right] =0, \quad i \neq j, ~~ i,j \neq 1.
\end{eqnarray*}
If $i=1\neq j$, then we see from Lemma \ref{lem:yamada_watanabe} that
\begin{eqnarray*}
\bE\left[\int_0^T  \left( \mu^{*}_t-\xi^{*,1}_t \right)\left( \mu^{*}_t-\xi^{*,j}_t \right) \,dt\Bigg| \mathcal X^1=x^1\right]
= \bE \left[ \int_0^T \Psi(t,x^1,W^1,W^0) \left( \mu^{*}_t-\xi^{*,j}_t \right) dt \right]
\end{eqnarray*}
for some real-valued function $\Psi$. Using again conditional independence and \eqref{action=equilibrium} we see that this term vanishes as well. As a result,
\begin{equation} \label{eq:estim_nash_equil}
    \bE\left[\int_0^T\left(\mu^{*}_t-\frac{1}{N}\sum_{j=1}^N\xi^{*,j}_t\right)^2\,dt \Bigg| \mathcal X^1=x^1\right]  \leq  \frac{2(M(x^1)+(2N-1)C)}{N^2} .
\end{equation}

We are now ready to the prove the $\epsilon$-equilibrium property of $\vec{\xi^{*}}$. By \eqref{eq:expected_L2_bound}, we have that $\xi^{*,1}\in\mathcal{A}^1$. For a given strategy $\xi\in\mathcal{A}^1$, let $X^\xi$ be the corresponding state process and let $J^1(\cdot;\mu^*)$ be Player 1's cost function when the average trading rate is replaced by the mean-field equilibrium. By Proposition \ref{verification-PVE}, $J^1(\xi;\mu^*)\geq J^1(\xi^{*,1};\mu^*)$, which implies:
\begin{equation*}
\begin{split}
&~J^{N,1}(\xi,\xi^{*,2},\cdots,\xi^{*,N})-J^{N,1}(\xi^{*,1},\cdots,\xi^{*,N})\\
\geq&~\bE\left[ \int_0^T\left(\kappa^1_t\left(\frac{1}{N}\sum_{j=2}^N\xi^{*,j}_t+\frac{1}{N}\xi_t\right)X^\xi_t+\eta^1_t\xi^2_t+\lambda^1_t (X^\xi_t)^2\right)\,dt\Bigg| \mathcal X^1=x^1\right] \\
&~-\bE\left[\int_0^T\left(\kappa^1_t\mu^{*}_t X^\xi_t+\eta^1_t (\xi_t)^2+\lambda^1_t (X^\xi_t)^2\right)\,dt\Bigg| \mathcal X^1=x^1\right]\\
&~+\bE\left[\int_0^T\left(\kappa^1_t\mu^{*}_t X^{*,1}_t+\eta^1_t(\xi^{*,1}_t)^2+\lambda^1_t (X^{*,1}_t)^2\right)\,dt\Bigg| \mathcal X^1=x^1\right] \\
&~-\bE\left[\int_0^T\left(\kappa^1_t\frac{1}{N}\sum_{j=1}^N\xi^{*,j}_tX^{*,1}_t+\eta^1_t(\xi^{*,1}_t)^2+\lambda^1_t (X^{*,1}_t)^2\right)\,dt\Bigg| \mathcal X^1=x^1\right]\\
:=&~I_1+I_2.
\end{split}
\end{equation*}
For the first difference $I_1$, using \eqref{eq:estim_nash_equil} we have that
\begin{equation*}
    \begin{split}
    & \sup_{\xi\in\mathcal{A}^1}|I_1| \\
     \leq&~\frac{\kappamax}{N}\sup_{\xi\in\mathcal{A}^1}\bE\left[\int_0^T|X^\xi_t||\xi_t|\,dt\Bigg|  \mathcal X^1=x^1\right] +\kappamax\sup_{\xi\in\mathcal A^1}\bE\left[\int_0^T|X^\xi_t|\left|\frac{1}{N}\sum_{j=2}^N\xi^{*,j}_t-\mu^{*}_t\right|\,dt\Bigg| \mathcal X^1=x^1\right]\\
     \leq&~\frac{\kappamax}{N}\sup_{\xi\in\mathcal{A}^1}\left(\bE\left[\int_0^T|X^\xi_t|^2\,dt\Bigg| \mathcal X^1=x^1\right]\right)^{\frac{1}{2}}\sup_{\xi\in\mathcal{A}^1}\left(\bE\left[\int_0^T|\xi_t|^2\,dt\Bigg| \mathcal X^1=x^1\right]\right)^{\frac{1}{2}}\\
     &~+\kappamax\sup_{\xi\in\mathcal A^1}\left(\bE\left[\int_0^T|X^\xi_t|^2\,dt\Bigg| \mathcal X^1=x^1\right]\right)^{\frac{1}{2}}\left(\bE\left[\int_0^T\left|\frac{1}{N}\sum_{j=1}^N\xi^{*,j}_t-\mu^{*}_t-\frac{1}{N}\xi^{*,1}_t\right|^2\,dt\Bigg| \mathcal X^1=x^1\right]\right)^{\frac{1}{2}}\\
    \leq&~ \frac{M(x^1)\kappamax T}{N}+\frac{2\kappamax T\sqrt{M(x^1)}}{N} \left( \sqrt{M(x^1)} + \sqrt{ 2(M(x^1)+(2N-1)C)}\right).
    \end{split}
\end{equation*}
For the second difference $I_2$, again using \eqref{eq:estim_nash_equil}, we have that
\begin{equation*}
    \begin{split}
    I_2\leq&~ \kappamax\left(\bE\left[\left.\int_0^T|X^{*,1}_t|^2\,dt\right| \mathcal X^1=x^1\right]\right)^{\frac{1}{2}}\left(\bE\left[\left.\int_0^T\left|\mu^{*}_t-\frac{1}{N}\sum_{j=1}^N\xi^{*,j}_t\right|^2\,dt\right| \mathcal X^1=x^1\right]\right)^{\frac{1}{2}}\\
    \leq&~\frac{2\kappamax T^{}\sqrt{M (x^1)} \sqrt{ (M(x^1)+(2N-1)C)}}{N}
    \end{split}
\end{equation*}
This proves the assertion.
\end{proof}
%%%%%%%%%%%%%%%%%%%%%%%%%%%%%%%%%%%%%%%%%%%%
%%%%%%%%%%%%%%%%%%%%%%%%%%%%%%%%%%%%%%%%%%%%
%%%%%%%%%%%%%%%%%%%%%%%%%%%%%%%%%%%%%%%%%%%%

\section{Approximation by unconstrained MFGs}\label{sec:convergence}

In this section, we prove that the solution to our singular MFG can be approximated by the solutions to non-singular MFGs under additional assumptions on the market impact parameter. Specifically, we consider the following unconstrained MFGs:
\begin{equation}\label{eq-penalized-problem}
    \left\{\begin{aligned}
    1.&\textrm{ fix a process }\mu;\\
     2.&\textrm{ solve the standard optimization problem: minimize}\\
&J^n(\xi;\mu)=\bE\left[\int_0^T\left(\kappa_t\mu_tX_t+\eta_t\xi^2_t+\lambda_tX^2_t\right)\,dt+nX^2_T\Bigg| \mathcal X\right]\\
      &\textrm{  such that }       dX_t=-\xi_t\,dt,\quad X_0=\mathcal X;\\
	3. & \textrm{ search for the fixed point }
   \mu_t=\bE[\xi^{*,\mathcal X}_t|\cF^0_t],\textrm{ for }a.e.~ t\in[0,T].
    \end{aligned}\right.
\end{equation}

We will need the following assumption on the solution $A\in \cM_{-1}$ to the first equation in \eqref{Riccati-system-PVE} with the terminal condition $+\infty$. It implies in particular that $X^*\in\mathcal H_1$.

\begin{ass}\label{ass-app}
There exists a constant $C$ such that for any $0\leq r\leq s < T$
$$\exp\left(-\int_r^s \frac{A_u}{2\eta_u} du \right) \leq C \left( \frac{T-s}{T-r} \right) . $$
\end{ass}

The following result is proven in the appendix.

{
\begin{lemma} \label{rem:suff_cond_ass-app}
Assumption \ref{ass-app} holds under each of the following conditions:
\begin{itemize}
\item $\eta$ is deterministic;
\item $1/\eta$ is a positive martingale;
\item $1/\eta$ has uncorrelated multiplicative increments, namely for any $0\leq s \leq t$
$$\bE \left[ \frac{\eta_s}{\eta_t} \bigg| \cF_s \right] = \bE \left[ \frac{\eta_s}{\eta_t}  \right] .$$
\end{itemize}
\end{lemma}
}
Using the same arguments as in Section \ref{sec:results}, the unconstrained control problem leads to the following conditional mean field FBSDE:
\begin{equation} \label{penalized_FBSDE2}
	\left\{\begin{aligned}
	dX^n_t=&~ \left(-\frac{A^n_t X^n_t+B^n_t}{2\eta_t}\right)\,dt, \\
	-dB^n_t =&~\left(-\frac{A^n_t B^n_t}{2\eta_t}+{\kappa_t} \bE\left.\left[\frac{A^n_t X^n_t+B^n_t}{2\eta_t}\right|\mathcal F^0_t\right]\right)\,dt-Z^{B^n}_t\,d\widetilde{W}_t, \\
dY^n_t=&~\left(-2\lambda_tX^n_t-\kappa_t\bE\left[\frac{A^n_tX^n_t+B^n_t}{2\eta_t}\bigg| \mathcal F^0_t\right]\right)\,dt+Z^{Y^n}_t\,d\widetilde W_t,\\
	 X^n_0 =&~\mathcal X,\\
	 B^n_T =&~0,\\
    Y^n_T=&~2nX^n_T,
	\end{aligned}\right.
\end{equation}
where
\begin{equation} \label{penalized_BSDE2}
	\left\{\begin{aligned}
	-dA^n_t&= \left\{2\lambda_t-\frac{(A^n_t)^2}{2\eta_t}\right\}\,dt-Z^{A^n}_t\,d\widetilde{W}_t, \\
	A^n_T&=2n.
	\end{aligned}\right.
\end{equation}
The existence of a solution $(A^n,Z^{A^n})$ to the BSDE \eqref{penalized_BSDE2}  can be deduced from Lemma \ref{lem:estim_H_n_1}. By the same lemma the sequence $\{A^n\}$ is a non-decreasing sequence converging pointwise to $A$ and there exists a constant $\fC>0$ such for any $n$,
\[
 \|A^n\|_{\cM_{-1}}  \leq \|A^n\|_{\cM^n_{-1}}\leq \fC,
\]
where the space $\cM^n_{l}$ is defined as
\[
    \cM^n_{l} := \left\{ U \in \cP_{\bF}([0,T]\times \Omega;\mathbb R\cup\{\infty\}): \ \left(T-.+\frac{\etamin}{n}\right)^{-l} U_\cdot \in L^\infty_\bF([0,T]\times \Omega;\mathbb R\cup\{\infty\}) \right\},
\]
and endowed with the norm
\[
	\|U\|_{\cM^n_l}:=\esssup_{(t,\omega)\times[0,T]\times\Omega}\frac{|U_t|}{\left(T-t+\frac{\etamin}{n}\right)^l}.
\]
We shall also need the following analogs to the space $\cH_\nu$:
\[
    \cH^n_{l}:=\left\{ U \in \cP_{\bF}([0,T]\times \Omega;\mathbb R\cup\{\infty\}): \ \left(T-.+\frac{\etamin}{n}\right)^{-l} U_\cdot \in S^2_\bF([0,T]\times \Omega;\mathbb R\cup\{\infty\}) \right\},
\]
endowed with the norm
\[
    \|U\|_{n,l}:=\left(\bE\left[\sup_{0\leq t\leq T} \left|\frac{U_t}{\left(T-t+\frac{\etamin}{n}\right)^{l}}\right|^2\right]\right)^{\frac{1}{2}}.
\]

The next result can be obtained using similar arguments as in the proof of Theorem \ref{existence_MF_FBSDE}. In fact, we have a slightly stronger result.
\begin{theorem}\label{existence-penalization}
Assume that Assumption \ref{ass-PVE} holds and that $\cX$ is a square integrable random variable.
Then, for any fixed $\fp\in[0,1]$ and $f\in L^2_{\bF}([0,T]\times\Omega;\mathbb R)$, there exists a unique solution $$(X^n,B^n,Y^n,Z^{B^n},Z^{Y^n})\in \mathcal H^n_{\alpha}\times \mathcal H^n_{\gamma}\times S^2_{\bF}([0,T]\times\Omega;\bR)\times L^2_{\bF}([0,T]\times\Omega;\mathbb R^m)\times L^2_{\bF}([0,T]\times\Omega;\mathbb R^m)$$ to the following FBSDE system:
\begin{equation}\label{conditional-MF-FBSDE-3-n}
\left\{\begin{aligned}
dX^n_t =&~-\frac{1}{2\eta_t} (A^n_t X^n_t + B^n_t)\,dt,\\
    -dB^n_t =&~\left(\kappa_t \fp \bE\left[\left.\frac{1}{2\eta_t}\left(A^n_tX^n_t+B^n_t\right)\right|\cF^0_t\right]+f_t-\frac{A^n_tB^n_t}{2\eta_t}\right)\,dt-Z^{B^n}_t\,d\widetilde{W}_t,\\
dY^n_t=&~\left(-2\lambda_tX^n_t-\kappa_t\fp \bE\left[\left.\frac{A^n_tX^n_t+B^n_t}{2\eta_t}\right|\mathcal F^0_t\right]-f_t\right)\,dt+Z^{Y^n}_t\,d\widetilde W_t,\\
X^n_0 =&~\mathcal X \\
B^n_T =&~ 0,\\
Y^n_T=&~2nX^n_T.
\end{aligned}
\right.
\end{equation}
\end{theorem}
\begin{proof}
The proof is similar to that of Theorem \ref{existence_MF_FBSDE}. We only need to note that by Lemma \ref{lem:estim_H_n_1},
\[
    \left.e^{-\int_s^t\frac{A^n_r}{2\eta_r}\,dr}\leq \left( \frac{T-t+\frac{\etamin}{n}}{T-s+\frac{\etamin}{n}} \right)^\alpha\right..
\]
\end{proof}

In order to establish the convergence of the value functions of the unconstrained problems to the value function of the constrained problem we need a uniform norm estimate for the sequence $(X^n,B^n,Y^n)$.
\begin{lemma} \label{lem:estim_H_X_B_A_n}
Let Assumption \ref{ass-PVE}  hold.
There exists a constant $\overline{\fC}>0$ such that
\begin{equation}\label{uni-bound}
    \|X^n\|_{n,\alpha} + \|B^n\|_{n,\gamma}+\bE\left[\int_0^T|Y^n_t|^2\,dt\right]\leq \overline{\fC},
\end{equation}
for any $n$
where $(X^n,B^n,Y^n)$ is the unique solution to \eqref{penalized_FBSDE2}.
\end{lemma}
\begin{proof}
The proof is split into three steps.
\paragraph{Step 1.} When $\fp=0$ in \eqref{conditional-MF-FBSDE-3-n}, there exists $R\in\mathbb R$ independent of $n$ such that
\[
    \|X^n\|_{n,\alpha} + \|B^n\|_{n,\gamma}+\left(\bE\left[\int_0^T|Y^n_t|^2\,dt\right]\right)^{\frac{1}{2}}\leq R.
\]
This bound follows from modifications of arguments given in the proof of Lemma \ref{induction-basis-eta}. In fact,
\[
    \|B^n\|_{n,\gamma} \leq  \|B^n\|_{\gamma} \leq C \|f\|_{L^2} \leq R_1.
\]
Moreover,
\[
    |X^n_t|\leq \frac{|\mathcal X|(T-t+\frac{\etamin}{n})^\alpha}{(T+\frac{\etamin}{n})^\alpha}+C\int_0^t |B^n_s| \left( \frac{T-t+\frac{\etamin}{n}}{T-s+\frac{\etamin}{n}}\right)^\alpha\,ds.
\]
This implies $\|X^n\|_{n,\alpha} \leq R_2.$ Finally, {by analogy to the proof of Lemma \ref{induction-step-eta}, doing integration by part for $X^nY^n$, we have}
\[
    \bE\left[\int_0^T(Y^n_t)^2\,dt\right]\leq C\bE\left[\int_0^T(X^n_t)^2\,dt\right]+C\bE\left[\int_0^Tf_t^2\,dt\right] \leq R_3.
\]
\paragraph{Step 2.} Suppose that for some $\fp\in[0,1]$, the solution to \eqref{conditional-MF-FBSDE-3-n} satisfies
\[
    \|X^n\|_{n,\alpha} + \|B^n\|_{n,\gamma}+\left(\bE\left[\int_0^T|Y^n_t|^2\,dt\right]\right)^{\frac{1}{2}}\leq kR,
\]
for some $k\geq1$ independent of $n$. Then there exists $\fd>0$ independent of $\fp$ such that the solution  $(\widetilde X^n,\widetilde B^n,\widetilde Y^n)$ to \eqref{conditional-MF-FBSDE-3-n} with $\fp$ replaced by $\fp+\fd$ satisfies the same estimate for some $K>k$:
\begin{equation} \label{eq-step2}
    \|\widetilde X^n\|_{n,\alpha} + \|\widetilde B^n\|_{n,\gamma}+\left(\bE\left[\int_0^T|\widetilde Y^n_t|^2\,dt\right]\right)^{\frac{1}{2}}\leq KR.
\end{equation}
To prove this assertion, we introduce for any given $Y^n,f\in L^2_{\bF}([0,T]\times\Omega;\mathbb R)$ the FBSDE system
\begin{equation}\label{conditional-MF-FBSDE-4-n}
\left\{\begin{aligned}
d\widetilde X^n_t =&~-\frac{1}{2\eta_t} (A^n_t\widetilde X^n_t +\widetilde B^n_t)\,dt,\\
    -d\widetilde B^n_t =&~\left(\kappa_t \fp \bE\left[\left.\frac{1}{2\eta_t}\left(A^n_t\widetilde X^n_t+\widetilde B^n_t\right)\right|\cF^0_t\right]+\kappa_t \fd \bE\left[\left.\frac{Y^n_t}{2\eta_t}\right|\cF^0_t\right]+f_t-\frac{A^n_t\widetilde B^n_t}{2\eta_t}\right)\,dt-Z^{\widetilde B^n}_t\,d\widetilde{W}_t,\\
d\widetilde Y^n_t=&~\left(-2\lambda_t\widetilde X^n_t-\kappa_t\fp \bE\left[\left.\frac{A^n_t\widetilde X^n_t+\widetilde B^n_t}{2\eta_t}\right|\mathcal F^0_t\right]-\kappa_t\fd \bE\left[\left.\frac{Y^n_t}{2\eta_t}\right|\mathcal F^0_t\right]-f_t\right)\,dt+Z^{\widetilde Y^n}_t\,d\widetilde W_t,\\
\widetilde X^n_0 =&~\mathcal X, \\
\widetilde B^n_T =&~ 0,\\
\widetilde Y^n_T=&~2n\widetilde X^n_T.
\end{aligned}
\right.
\end{equation}
Arguing as in the proof of Theorem \ref{existence-penalization}, there exists a unique solution to \eqref{conditional-MF-FBSDE-4-n}. This defines a mapping
\[
    \Gamma: Y^n\rightarrow \widetilde Y^n.
\]
on $L^2_{\bF}([0,T]\times\Omega;\mathbb R)$. We now show the $\Gamma$ has a unique fixed point and that this fixed point belongs to $B^{L^2}_{2kR}(0)$, the subset of $L^2_{\bF}([0,T]\times\Omega;\mathbb R)$ such that the $L^2$-norm is bounded by $2kR$.

By the same arguments as in the proof as Lemma \ref{induction-step-eta} we have
\[
    \bE\left[\int_0^T|\Gamma(Y^n)(t)-\Gamma(\overline Y^{n})(t)|^2\,dt\right]\leq C\fd\bE\left[\int_0^T|Y^n_t-\overline Y^{n}_t|^2\,dt\right]\leq \frac{1}{4}\bE\left[\int_0^T|Y^n_t-\overline Y^{n}_t|^2\,dt\right],
\]
where $C$ does not depend on $n$ and $\fd$ is small enough but independent of $\fp$ and of $n$. Taking $\overline Y^{n}=0$, we have
\[
    \left(\bE\left[\int_0^T|\widetilde Y^n_t|^2\,dt\right]\right)^{\frac{1}{2}}\leq \frac{1}{2}\left(\bE\left[\int_0^T|Y^n_t|^2\,dt\right]\right)^{\frac{1}{2}}+\left(\bE\left[\int_0^T|\Gamma(0)(t)|^2\,dt\right]\right)^{\frac{1}{2}}.
\]
Note that $\Gamma(0)$ corresponds to the solution to \eqref{conditional-MF-FBSDE-3-n} with $\fp$. By assumption,
\[
    \left(\bE\left[\int_0^T|\Gamma(0)(t)|^2\,dt\right]\right)^{\frac{1}{2}}\leq kR.
\]
Thus, if we assume $Y^n\in B^{L^2}_{2kR}$,
\[
    \left(\bE\left[\int_0^T|\widetilde Y^n_t|^2\,dt\right]\right)^{\frac{1}{2}}\leq 2kR.
\]
This implies that $\Gamma$ is a mapping from $B^{L^2}_{2kR}(0)$ to itself. Since $B^{L^2}_{2kR}(0)$ is a Banach space the unique fixed point belongs to $B^{L^2}_{2kR}(0)$. This yields the desired $L^2$ estimate for $\widetilde Y$.

 Let $(\widetilde X^n,\widetilde B^n)$ be the solution corresponding to $\widetilde Y^n$ and $\fp+\fd$. Then, by H\"older's inequality,
 \begin{equation*}
    \begin{split}
        \frac{\left|\widetilde B^n_t\right|}{(T-t)^\gamma}\leq
        &~\frac{1}{(T-t)^\gamma}\bE\left[\left.\int_t^T\kappa_s(\fp+\fd)\bE\left[\left.\frac{|\widetilde Y^n_s|}{2\eta_s}\right|\mathcal F^0_s\right]\,ds\right|\mathcal F_t\right]\\
        \leq
        &~\frac{\kappamax}{2\etamin}\left(\bE\left[\left.\int_t^T\bE\left[\left.|\widetilde Y^n_s|^{\frac{1}{1-\gamma}}\right|\mathcal F^0_s\right]\,ds\right|\mathcal F_t\right]\right)^{1-\gamma}.
    \end{split}
 \end{equation*}
 Doob's maximal inequality yields that
 \begin{equation*}
    \begin{split}
        &~\bE\left[\sup_{0\leq t\leq T}\left|\frac{\widetilde B^n_t}{(T-t)^\gamma}\right|^2\right]\\
        \leq
        &~\frac{(\kappamax)^2}{4\etamin^2}\bE\left[\sup_{0\leq t\leq T}\left(\bE\left[\left.\int_t^T\bE\left[\left.|\widetilde Y^n_s|^{\frac{1}{1-\gamma}}\right|\mathcal F^0_s\right]\,ds\right|\mathcal F_t\right]\right)^{2(1-\gamma)}\right]\\
        \leq&~C\bE\left[\int_0^T|\widetilde Y^n_t|^2\,dt\right].
    \end{split}
 \end{equation*}
 Hence,
 \[
    \|\widetilde B^n\|_{n,\gamma}\leq C\left(\bE\left[\int_0^T|\widetilde Y^n_t|^2\,dt\right]\right)^{\frac{1}{2}}\leq CR
 \]
and
 \[
    \|\widetilde X^n\|_{n,\alpha}\leq C R.
 \]
 \paragraph{Step 3.} Since $\fd$ is independent of $\fp$, by iteration for only finitely many times, we have the solution for \eqref{penalized_FBSDE2} with $\fp=1$ and $f=0$ with the uniform estimate \eqref{uni-bound}.
\end{proof}

Under Assumption \ref{ass-app}, the value $\alpha$ appearing in the estimate of Theorem \ref{existence_MF_FBSDE} is equal to one. That is
\begin{equation}\label{est-X-app}
    \|X^*\|_1<\infty.
\end{equation}
This allows us to prove the convergence of the optimal position and control.

\begin{lemma}\label{app-control-position}
Let $(X,B,Y)$ be the solution to the FBSDEs \eqref{conditional-MF-FBSDE} and \eqref{conditional-MF-FBSDE_2} (Proposition \ref{existence_MF_FBSDE}). Under Assumption \ref{ass-PVE} and Assumption \ref{ass-app},
\[
 \lim_{n\to +\infty} \left\{  \bE\left[\int_0^T|X^n_t-X_t|^2\,dt\right]+    \bE\left[\int_0^T|B^n_t-B_t|^2\,dt\right]
+    \bE\left[\int_0^T|Y^n_t-Y_t|^2\,dt\right] \right\}=  0.
\]
\end{lemma}
\begin{proof}
Using the same arguments as in the proof of Lemma \ref{induction-step-eta}, we have for each $\epsilon>0$
\begin{equation}\label{eq:L2_estim_penal_approx}
\begin{split}
&~\bE\left[\int_0^{T-\epsilon}|Y^n_t-Y_t|^2\,dt\right]+\bE\left[\int_0^{T-\epsilon}|X^n_t-X_t|^2\,dt\right]\\
\leq&~C\bE\left[|(B^n_{T-\epsilon}-B_{T-\epsilon})(X^n_{T-\epsilon}-X_{T-\epsilon})|\right]+ C \bE\left[|(A^n_{T-\epsilon}-A_{T-\epsilon})X_{T-\epsilon}(X^n_{T-\epsilon}-X_{T-\epsilon})|\right].
\end{split}
\end{equation}
The two terms in the above summation admit the following estimates
\begin{equation*}
    \begin{split}
    &~\bE[|(B^n_{T-\epsilon}-B_{T-\epsilon})(X^n_{T-\epsilon}-X_{T-\epsilon})|]\\
    \leq&~C\bE[|B^n_{T-\epsilon}|^2]+C\bE[|B_{T-\epsilon}|^2]+C\bE[|X^n_{T-\epsilon}|^2]+C\bE[|X_{T-\epsilon}|^2]\\
    \leq&~C\left(\epsilon+\frac{1}{n}\right)^{2\gamma}\|B^n\|^2_{n,\gamma}+C\left(\epsilon+\frac{1}{n}\right)^{2}\|X^n\|^2_{n,{\alpha}}+C\epsilon^{2\gamma}\|B\|^2_{\gamma}+C\epsilon^2\|X\|^2_{1},
     \end{split}
\end{equation*}
respectively,
\begin{equation*}
    \begin{split}
    &~\bE[|(A^n_{T-\epsilon}-A_{T-\epsilon})X_{T-\epsilon}(X^n_{T-\epsilon}-X_{T-\epsilon})|]\\
    \leq&~C\bE\left[\sup_{0\leq t\leq T}\left|\frac{X_t}{T-t}\right|\left(\sup_{0\leq t\leq T}\frac{|X^n_t|}{\left(T-t+\frac{\etamin}{n}\right)^{{\alpha}}}\right)\left(\epsilon+\frac{1}{n}\right)^{{\alpha}}+\epsilon\sup_{0\leq t\leq T}\frac{|X_t|}{T-t}\right]  \quad(\textrm{by Lemma }\ref{existence-asymptotic-expansion-A}\textrm{ and Lemma }\ref{lem:estim_H_n_1})\\
    \leq&~C\left[\left(\epsilon+\frac{1}{n}\right)^{{\alpha}}+\epsilon\right](\|X^n\|^2_{n,\alpha}+\|X\|^2_{1})\\
    \leq&~C\left[\left(\epsilon+\frac{1}{n}\right)^{{\alpha}}+\epsilon\right] \quad(\textrm{by Lemma \ref{lem:estim_H_X_B_A_n} and \eqref{est-X-app}}).
    \end{split}
\end{equation*}
Letting $\epsilon$ go to zero in \eqref{eq:L2_estim_penal_approx}, by Theorem \ref{existence_MF_FBSDE} and Lemma \ref{lem:estim_H_X_B_A_n} we get
\[
    \bE\left[\int_0^{T}|Y^n_t-Y_t|^2\,dt\right]+\bE\left[\int_0^{T}|X^n_t-X_t|^2\,dt\right]\leq C\left(\frac{1}{n}\right)^{2\gamma}+C\left(\frac{1}{n}\right)^{2}+\frac{C}{n^\alpha}.
\]
Hence we obtain the desired limit for $(Y^n-Y)$ and $(X^n-X)$. By the expression for $B$, we have
\begin{equation*}
\begin{split}
    |B^n_t-B_t|\leq&~ \bE\left[\int_t^Te^{-\int_t^s\frac{A^n_r}{2\eta_t}\,dr}\kappa_s\bE\left[\frac{|Y^n_s-Y_s|}{2\eta_s}\bigg|\mathcal F^0_s\right]\,ds\bigg|\mathcal F_t\right]\\
    &~+\bE\left[\int_t^T\left|1-e^{-\int_t^s\frac{(A_r-A^n_r)}{2\eta_t}\,dr}\right|\kappa_s\bE\left[\frac{|Y_s|}{2\eta_s}\bigg|\mathcal F^0_s\right]\,ds\bigg|\mathcal F_t\right].
\end{split}
\end{equation*}
Let us recall that $\{A^n\}$ is a non-decreasing sequence converging to $A$. This leads to
\[
    \bE\left[\int_0^{T}|B^n_t-B_t|^2\,dt\right]\rightarrow 0.
\]
\end{proof}
Let us denote by $V^n(\mathcal X;\mu^n)$ the value function associated with the penalized problem \eqref{eq-penalized-problem}. The next theorem shows the convergence of $V^n(\mathcal X;\mu^n):=V^n(\mathcal X)$ to the value function $V(\mathcal X;\mu^*):=V(\mathcal X)$ associated with the constrained MFG.

\begin{theorem}
Under Assumption \ref{ass-PVE} and Assumption \ref{ass-app}, the value function $V^n(\mathcal X)$ converges to $V(\mathcal X)$ in $L^1(\Omega)$.
\end{theorem}
\begin{proof}
Let $(X^n,Y^n,B^n)$ be the solution in Theorem \ref{existence-penalization}. Recall that $\xi^{n,*}=\frac{Y^n}{2\eta}$ is the optimal strategy for the penalized problem with degree $n$ and the related optimal process $X^{n,*}$ is equal to $X^n$. Thus, with $\left(\mu^{n,*}=\mathbb E\left[\left.\frac{Y^n_t}{2\eta_t}\right|\mathcal F^0_t\right]\right)_{0\leq t\leq T}$ fixed, the optimal strategy $\xi^*$ for the constraint optimization is an admissible control for the penalized optimization. We denote by $X^*=X$ the optimal state process related to $\xi^*$ (Proposition \ref{verification-PVE} and Equation \eqref{eq:optimal_candidates}). 
Let us define
$$\Delta_n = \mathbb E\left[\int_0^
    T\kappa_s\mathbb E\left(\left.\frac{Y_s-Y^n_s}{2\eta_s}\right|\mathcal F^0_s\right)X^*_s\,ds\Bigg| \mathcal X \right].$$
From Lemma \ref{app-control-position}, $\displaystyle \lim_{n\to +\infty} \bE (|\Delta_n| )= 0.$
Recalling that
$$J(\mathcal X,\xi^{n,*};\mu^{n,*}) =   \bE\left[\int_0^{T}\left(\kappa_s\mu^{n,*} X^{n,*}_s+\eta_s(\xi^{n,*}_s)^2+\lambda_s(X^{n,*}_s)^2\right)\,ds\Bigg| \mathcal X \right],$$
we have
\begin{equation*} \label{eq:conv_val_fct_PVE}
    \begin{split}
    V(\mathcal X) = &~\bE\left[\int_0^{T}\kappa_s\mu^* X^*_s+\eta_s(\xi^*_s)^2+\lambda_s(X^*_s)^2\,ds\Bigg| \mathcal X \right]\\
    =&~\bE\left[\int_0^{T}\kappa_s\mu^{n,*} X^*_s+\eta_s(\xi^*_s)^2+\lambda_s(X^*_s)^2\,ds\Bigg| \mathcal X \right]+\Delta_n\\
  \geq&~  \bE\left[\int_0^{T}\left(\kappa_s\mu^{n,*} X^{n,*}_s+\eta_s(\xi^{n,*}_s)^2+\lambda_s(X^{n,*}_s)^2\right)\,ds + n(X^{n,*}_T)^2\Bigg| \mathcal X \right] +\Delta_n \\
     =&~V^n(\mathcal X)+\Delta_n 
   \geq  J(\mathcal X,\xi^{n,*};\mu^{n,*}) + \Delta_n.
     \end{split}
\end{equation*}
Hence we deduce that
$$V(\mathcal X) - J(\mathcal X,\xi^{n,*};\mu^{n,*})  \geq V(\mathcal X) - V_ n(\mathcal X) \geq \Delta_n,$$
thus
$$ |V(\mathcal X) - V_ n(\mathcal X)| \leq  |\Delta_n| + |V(\mathcal X) - J(\mathcal X,\xi^{n,*};\mu^{n,*}) | .$$
Again by Lemma \ref{app-control-position},
$$\lim_{n\to +\infty} \bE  |V(\mathcal X) - J(\mathcal X,\xi^{n,*};\mu^{n,*}) | = 0.$$

\end{proof}

\begin{remark}
As a by-product of the proof, we get that $\displaystyle \lim_{n \to +\infty} \bE \left[ n (X^{n,*}_T)^2\right] = 0$.
Moreover
\[
    |X^{n,*}_T|\leq \frac{C}{n}\left(|\mathcal X|+\sup_{0\leq t\leq T}\frac{|B^n_t|}{\left(T-t+\frac{\etamin}{n}\right)^\gamma}\right)\rightarrow 0\quad\textrm{ a.s.}.
\]
\end{remark}

The proof of convergence of the value function simplifies substantially under the common information assumption (Subsection \ref{sec:common-initial}). {In particular, Assumption \ref{ass-app} is not necessary here.} In this case,
$Y^n = A^n X^n$ where
\[
    -dA^n_t=\left(2\lambda_t+\frac{\kappa_tA^n_t}{2\eta_t}-\frac{(A^n_t)^2}{2\eta_t}\right)\,dt-Z^{A^n}_tdW^0_t, ~~A^n_T=2n
\]
and
\[
    dX^n_t=-\frac{A^n_t X^n_t}{2\eta_t}\,dt, \ X_0=x\in\mathbb R
\]
The optimal strategy and the resulting portfolio process are given by, respectively,
\begin{equation*}
    \xi^{n,*}_t=\mu^{n,*}_t = \frac{A^n_t X^n_t}{2\eta_t},\quad X_t^{n,*}=X^n_t=xe^{-\int_0^t\frac{A^n_r}{2\eta_r}\,dr}\qquad t\in[0,T].
\end{equation*}
Since the sequence $A^n$ is non-decreasing and converges to $A$, we deduce that $X^{n,*}$ converges to $X^*$ a.s. and that $\xi^{n,*}$ converges to $\xi^*$ a.e. a.s..
Moreover, for fixed $\mu^{n,*}$, $\xi^*$ is suboptimal to the penalized optimization. This implies that
\begin{equation*}
    \begin{split}
   &~ \bE\left[\left.\int_0^{T}\left(\kappa_s\xi^*_sX^*_s+\eta_s(\xi^*_s)^2+\lambda_s(X^*_s)^2\right)\,ds\right.\right]\\
   =&~\bE\left[\left.\int_0^{T}\left(\kappa_s\xi^{n,*}_sX^*_s+\eta_s(\xi^*_s)^2+\lambda_s(X^*_s)^2\right)\,ds\right.\right]+\mathbb E\left[\int_0^T\kappa_sX^*_s(\xi^*_s-\xi^{n,*}_s)\,ds\right]\\
     \geq&~ \bE\left[\left.\int_0^{T}\left(\kappa_s\xi^{n,*}_sX^{n,*}_s+\eta_s(\xi_s^{n,*})^2+\lambda_s(X_s^{n,*})^2\right)\,ds+n(X^{n,*}_T)^2\right.\right]+\mathbb E\left[\int_0^T\kappa_sX^*_s(\xi^*_s-\xi^{n,*}_s)\,ds\right]\\
   \geq&~ \bE\left[\left.\int_0^{T}\left(\kappa_s\xi^{n,*}_sX^{n,*}_s+\eta_s(\xi_s^{n,*})^2+\lambda_s(X_s^{n,*})^2\right)\,ds \right.\right]+\mathbb E\left[\int_0^T\kappa_sX^*_s(\xi^*_s-\xi^{n,*}_s)\,ds\right].
    \end{split}
\end{equation*}
For any $\epsilon > 0$, it holds
\begin{equation*}
    \begin{split}
\lim_{n\to +\infty} &~\bE\left[\left.\int_0^{T-\epsilon}\left(\kappa_s\xi^{n,*}_sX^{n,*}_s+\eta_s(\xi_s^{n,*})^2+\lambda_s(X_s^{n,*})^2\right)\,ds \right.\right]\\
=&~ \bE\left[\left.\int_0^{T-\epsilon}\left(\kappa_s\xi^{*}_sX^*_s+\eta_s(\xi^*_s)^2+\lambda_s(X^*_s)^2\right)\,ds \right.\right].
    \end{split}
\end{equation*}
Hence, the monotone convergence theorem implies
\begin{equation}\label{CVE-convergence-1}
    \begin{split}
    \lim_{n\rightarrow\infty}&~\bE\left[\left.\int_0^{T}\left(\kappa_s\xi^{n,*}_sX^{n,*}_s+\eta_s(\xi_s^{n,*})^2+\lambda_s(X_s^{n,*})^2\right)\,ds+n(X^{n,*}_T)^2\right.\right]\\
    \geq&~\bE\left[\left.\int_0^{T}\left(\kappa_s\xi^{*}_sX^*_s+\eta_s(\xi^*_s)^2+\lambda_s(X^*_s)^2\right)\,ds \right.\right].
    \end{split}
\end{equation}
{
Moreover,
\[
    |\kappa_sX^*_s(\xi^*_s-\xi^{n,*}_s)|\leq \kappamax|x||\xi^*_s-\xi^{n,*}_s|,
\]
which is $L^2$ bounded uniformly in $n$, due to Lemma \ref{lem:estim_H_X_B_A_n}. Vitali convergence implies
\begin{equation}\label{CVE-convergence-2}
    \lim_{n\rightarrow\infty}\mathbb E\left[\int_0^T\kappa_sX^*_s(\xi^*_s-\xi^{n,*}_s)\,ds\right]=0.
\end{equation}
The convergence \eqref{CVE-convergence-1} and \eqref{CVE-convergence-2} yields the desired result.
%%%%%%%%%%%%%%%%%%%%%%%%%%%%%%%%%%%%%%%%%%%%%%%%%%%%%%%%%%%%%%%%%%%%%%%%%%%%%%%%%%
%%%%%%%%%%%%%%%%%%%%%%%%%%%%%%%%%%%%%%%%%%%%%%%%%%%%%%%%%%%%%%%%%%%%%%%%%%%%%%%%%%
%%%%%%%%%%%%%%%%%%%%%%%%%%%%%%%%%%%%%%%%%%%%%%%%%%%%%%%%%%%%%%%%%%%%%%%%%%%%%%%%%%

\begin{appendix}
\section{Appendix}

In this appendix we recall an existence of solutions result for a stochastic Riccati equation with singular terminal condition and prove  Lemma \ref{rem:suff_cond_ass-app}. We assume throughout that $\lambda$, $\eta$ and $1/\eta$ are bounded.

\subsection{Stochastic Riccati equations with singular terminal value}

\begin{lemma}\label{existence-asymptotic-expansion-A}\cite[Theorem 2.2]{AJK-2014}\cite[Theorem 6.1, Theorem 6.3]{GHS-2013}
In $S^2_{\bF}([0,T-]\times\Omega;\mathbb R)\times L^2_{\bF}([0,T-]\times\Omega;\mathbb{R}^{m})$ there exists a unique solution to
\begin{equation*}
\left\{\begin{aligned}
    -dA_t=&~\left(2\lambda_t-\frac{A_t^2}{2\eta_t}\right)\,dt-Z_t^A\,dW_t,\\
    A_T=&~\infty.
    \end{aligned}\right.
\end{equation*}
Moreover, there holds the following estimate
  \begin{equation}\label{asymptotic-expansion-A}
   \frac{1}{\bE\left[\left.\int_t^T\frac{1}{2\eta_s}\,ds\right|\mathcal{F}_t\right]}\leq A_t \leq\frac{1}{(T-t)^2}\bE\left[\left.\int_t^T2\eta_s+2(T-s)^2\lambda_s\,ds\right|\mathcal{F}_t\right].
\end{equation}
\end{lemma}

\begin{corollary}\label{cor-appendix}
The BSDE \eqref{Y-AX}
\begin{equation*}
\left\{\begin{aligned}
	-dA_t=& ~\left(2\lambda_t+\frac{\kappa_tA_t}{2\eta_t}-\frac{A_t^2}{2\eta_t}\right)\,dt-Z^A_tdW^0_t, \\
	A_T=\ & ~\infty.
   \end{aligned}\right.
\end{equation*}
has a unique solution.
\end{corollary}
\begin{proof}
Let $\widetilde{A}_t=A_t e^{\int_0^t\frac{\kappa_s}{2\eta_s}\,ds}$. Then,
\begin{equation}\label{Y-AX-2}
\left\{\begin{aligned}
    -d\widetilde{A}_t=& \left[2\lambda_te^{\int_0^t\frac{\kappa_s}{2\eta_s}\,ds}-\frac{\widetilde{A}^2_t}{2\eta_te^{\int_0^t\frac{\kappa_s}{2\eta_s}\,ds}}\right]\,dt-\widetilde{Z}_t\,dW^0_t, \\
    \widetilde{A}_T=& \infty.
       \end{aligned}\right.
\end{equation}
Hence, the assertion follows from the preceding lemma.
\end{proof}

\begin{lemma} \label{lem:estim_H_n_1}
For each $n$, there exists a unique solution $A^n$ to the BSDE
\begin{equation}\label{appendix-An}
\left\{\begin{aligned}
    -dA^n_t=&~\left(2\lambda_t-\frac{(A^n_t)^2}{2\eta_t}\right)\,dt-Z_t^{A^n}\,dW_t,\\
    A^n_T=&~2n.
    \end{aligned}\right.
\end{equation}
\[
    A^n_t\geq \frac{1}{\frac{1}{2n}+\bE\left[\left.\int_t^T\frac{1}{2\eta_s}\,ds\right|\mathcal F_t\right]}.
\]
Moreover, the sequence $A^n$ is non-decreasing and converges to $A$. There exists a constant $\fC$ such that for any $n$:
$$\|A^n\|_{\cM_{-1}} +\|A^n\|_{\cM^n_{-1}} \leq \fC.$$
\end{lemma}
\begin{proof}
The first and second assertions are results of \cite[Proposition 3.1,Theorem 3.2]{AJK-2014}, respectively. For any $t$, $n$ and $a$, we have
$$2\lambda_t-\frac{a^2}{2\eta_t} \leq 2\lambda_t -\frac{2}{\left( T-t+\frac{\etamin}{n}\right)} a + \frac{2\eta_t}{\left( T-t+\frac{\etamin}{n}\right)^2} = g(t,a).$$
Let us denote by $\Psi^n$ the solution of the BSDE with generator $g$ and terminal condition $2n$. By the comparison principle for BSDEs, we have $A^n_t \leq \Psi^n_t$ and by the solution formula for linear BSDEs,
\begin{equation*}
    \begin{split}
\Psi^n_t = &~\left( \frac{T+\frac{\etamin}{n}}{T-t+\frac{\etamin}{n}} \right)^2 \bE \left[ \left( \frac{\frac{\etamin}{n}}{T+\frac{\etamin}{n}} \right)^2 2n  + \int_t^T \left( \frac{T-s+\frac{\etamin}{n}}{T+\frac{\etamin}{n}} \right)^2\left.  \left( \frac{2\eta_s}{\left( T-s+\frac{\etamin}{n}\right)^2} +2\lambda_s \right)  \right| \cF_t  \right] \\
 = &~\frac{2\etamin^2}{n} \frac{1}{\left( T-t+\frac{\etamin}{n} \right)^2}
+ \frac{1}{\left(T-t+\frac{\etamin}{n}\right)^2}  \bE \left[ \left.\int_t^T  \left( 2\eta_s +2\left( T-s+\frac{\etamin}{n} \right)^2\lambda_s \right)  \right| \cF_t \right].
    \end{split}
\end{equation*}
Hence
\begin{equation*}
    \begin{split}
&~\left( T-t+\frac{\etamin}{n} \right)\Psi^n_t\\
 \leq &~ \frac{2\etamin^2}{\etamin+n(T-t)}  + \frac{1}{\left(T-t+\frac{\etamin}{n}\right)}  \bE \left[\left. \int_t^T  \left( 2\eta_s +2\left( T-s+\frac{\etamin}{n} \right)^2\lambda_s \right)  \right| \cF_t \right] \\
\leq &~ 2\etamin +  \frac{1}{T-t}  \bE \left[\left. \int_t^T  \left( 2\eta_s +2\left( T-s+\frac{\etamin}{n} \right)^2\lambda_s \right)  \right| \cF_t \right] =\fC.
    \end{split}
\end{equation*}
Thus $\left( T-t+\frac{\etamin}{n} \right)A^n_t  \leq \fC$, that is $\|A^n\|_{\mathcal M^n_{-1}} \leq \fC$.
\end{proof}

\subsection{On Assumption \ref{ass-app}}

{
 Assumption \ref{ass-app} states that there exists a constant $C$ such that a.s. for any $0\leq r\leq s < T$
\[
    \exp\left(-\int_r^s \frac{A_u}{2\eta_u} du \right) \leq C \left( \frac{T-s}{T-r} \right).
\]
The left-hand side is equal to the optimal state process $\chi$ of the control problem studied in \cite{AJK-2014,GHS-2013} with initial value equal to 1 at time $r$. In particular from the proof of \cite[Theorem 4.2]{AJK-2014}, the process $M$ defined on $[r,T)$ by
\[
    M_s = \frac{1}{A_r} \left[ A_s \chi_s + 2\int_r^s \lambda_u \chi_u du \right]
\]
is a non-negative local martingale with $M_r= 1$. Hence for any $s\in [r,T)$
\[
    \exp\left(-\int_r^s \frac{A_u}{2\eta_u} du \right)= \chi_s \leq \frac{A_r}{A_s}  M_s \leq \frac{\etamax + T \lambdamax }{\etamin} \ \left( \frac{T-s}{T-r} \right) M_s =C \ \left( \frac{T-s}{T-r} \right) M_s.
\]
Since $M$ is also a non-negative supermartingale $M_t$ converges almost surely as $t$ goes to $T$ and the limit $M_T$ satisfies $\bE (M_T) \leq 1$. Therefore Assumption \ref{ass-app} does not strike us as overly restrictive.
}

\textsc{Proof of Lemma \ref{rem:suff_cond_ass-app}}. From \eqref{asymptotic-expansion-A}
\[
    -\frac{A_u}{2\eta_u} \leq -\frac{1}{\bE\left[\left.\int_u^T\frac{\eta_u}{\eta_s}\,ds\right|\mathcal{F}_u\right]} = -\frac{1}{\int_u^T\bE\left[\left.\frac{\eta_u}{\eta_s}\right|\mathcal{F}_u\right]\,ds}.
\]
By the very definition of uncorrelated multiplicative increments for $1/\eta$ and from \cite[Lemma 5.1]{AJK-2014}
\[
    -\frac{A_u}{2\eta_u} \leq  -\frac{1}{\int_u^T\bE\left[\frac{\eta_u}{\eta_s}\right]\,ds} = -\frac{1}{\int_u^T\frac{\bE\left[1/\eta_s\right]}{\bE\left[1/\eta_u\right]}\,ds} = -\frac{\bE\left[1/\eta_u\right]}{ \int_u^T\bE\left[1/\eta_s\right]\,ds} = \frac{1}{N_u} dN_u
\]
with $N_u := \int_u^T\bE\left[1/\eta_s\right]\,ds$. Hence
\[
    \exp\left(-\int_r^s \frac{A_u}{2\eta_u} du \right) = \exp\left(\int_r^s  \frac{1}{N_u} dN_u \right) = \frac{N_s}{N_r} = \frac{\int_s^T\bE\left[1/\eta_v\right]\,dv}{\int_r^T\bE\left[1/\eta_v\right]\,dv} \leq \frac{\etamax}{\etamin} \left( \frac{T-s}{T-r} \right).
\]
If $1/\eta$ is a positive martingale, then again from \cite[Lemma 5.1]{AJK-2014}, we get that $1/\eta$ has uncorrelated multiplicative increments. If $\eta$ is deterministic, we have directly that
\[
    -\frac{A_u}{2\eta_u} \leq -\frac{1}{\eta_u \int_u^T\frac{1}{\eta_s}\,ds} = \frac{1}{N_u} \frac{dN_u}{du}.
\]
\hfill $\Box$
\end{appendix}

%%%%%%%%%%%%%%%%%%%%%%%%%%%%%%%%%%%%%%%%%%%%%%%%%%%%%%%%%%%%%%%%%%%%%%%%%%%%%%%%%%%
%%%%%%%%%%%%%%%%%%%%%%%%%%%%%%%%%%%%%%%%%%%%%%%%%%%%%%%%%%%%%%%%%%%%%%%%%%%%%%%%%%%%
%%%%%%%%%%%%%%%%%%%%%%%%%%%%%%%%%%%%%%%%%%%%%%%%%%%%%%%%%%%%%%%%%%%%%%%%%%%%%%%%%%%%
\section*{Acknowledgments.}
Financial support by the Berlin Mathematical School (BMS), the TRCRC 190 {\it Rationality and competition: the economic performance of individuals and firms} and the Tier 2 grant ``Nonstandard BSDEs in Mathematical Finance: Theory, Application, Numerical Methods'' is gratefully acknowledged. We thank participants of  the IPAM workshop ``Mean Field Games'' for valuable comments and suggestions. We also thank seminar participants at various places for helpful comments and discussions.

\bibliography{biblio}

\end{document}